\documentclass[10pt]{amsart}
\usepackage{amsmath, amsthm, amssymb, amsfonts, verbatim}
\usepackage{array}
\usepackage[bookmarks]{hyperref}
\usepackage[active]{srcltx}
\usepackage{amscd}
\usepackage{epsfig}
\usepackage{color}
\usepackage{comment}
\usepackage{setspace}
\usepackage{textcomp}
\usepackage{stmaryrd}
\usepackage{multirow}

\numberwithin{equation}{section}

\newcommand{\ra}{\rightarrow}

\newcommand{\half}{\frac{1}{2}}

\newcommand{\e}{\bs{\ul{\epsilon}}}

\newcommand{\R}{{\mathbb R}}

\newcommand{\grad}{\operatorname{grad}}

\renewcommand{\div}{\operatorname{div}}
\newcommand{\diag}{\operatorname{diag}}

\newcommand{\bs}{\boldsymbol}

\newcommand{\lap}{\Delta}
\newcommand{\pd}{\partial}

\newcommand{\ve}{\varepsilon}
\newcommand{\ul}{\underline}

\newcommand{\Ml}{\Bbb{M}_{\lambda}}
\newcommand{\Vl}{\Bbb{V}_{\lambda}}

\newcommand{\tblu}[1]{\textcolor{black}{#1}}
\newcommand{\tred}[1]{\textcolor{black}{#1}}
\newcommand{\tgreen}[1]{\textcolor{black}{#1}}
\newcommand{\algns}[1]{\begin{align*}#1\end{align*}}
\newcommand{\algn}[1]{\begin{align}#1\end{align}}

\newtheorem{theorem}{Theorem}[section]
 
\newtheorem{lemma}[theorem]{Lemma}
\newtheorem{cor}[theorem]{Corollary}

\newtheorem{rmk}[theorem]{Remark}
\newtheorem{eg}[theorem]{Example}

\begin{document}

\title[]{Parameter-robust discretization and preconditioning of Biot's consolidation model}\thanks{The works of Jeonghun~J.~Lee and Ragnar~Winther have been supported by the European Research Council under the European Union's Seventh Framework Programme (FP7/2007-2013) ERC grant agreement 339643. The work of Kent-Andre~Mardal has been supported by the Research Council of Norway through grant no. 209951 and a Center of Excellence grant awarded to the Center for Biomedical Computing at Simula Research Laboratory.}
\author{Jeonghun J. Lee, Kent-Andre Mardal and Ragnar Winther}
\maketitle

 \begin{abstract}
\tgreen{
Biot's consolidation model in poroelasticity has a number of applications in science, medicine, and engineering.  The model depends on various parameters, and in practical applications these parameters ranges over several orders of magnitude. A current challenge is to design discretization techniques and solution algorithms that are well behaved with respect to these variations. The purpose of this paper is to study finite element discretizations of this model and  construct  block diagonal preconditioners for the discrete Biot systems. The approach taken here is to consider
the stability of the problem in non-standard or weighted Hilbert spaces and employ the operator preconditioning approach.  
We derive preconditioners that are robust with respect to both the variations of the parameters and the mesh refinement. 
The parameters of interest are small time-step sizes, large bulk and shear moduli, and small hydraulic conductivity.  
}
\end{abstract}
%\dfn{Reviewer 1 : The two page reviewer, Reviewer 2 : the three pages reviewer, I use Rv1 and Rv2 to denote them and use C1, C2, ... to indicate the number of correction/minor comments by them. For example, Rv1-C1, Rv2-C1. For the specific comments of the reviewer 2, I use Rv2-SC1, Rv2-SC2, ...}
\section{Introduction}
Biot's consolidation model describes the deformation of an elastic porous medium and the viscous fluid flow inside when the porous medium is saturated by the fluid. 
The unknowns are the displacement of the elastic medium, 
$\bs{u}$, and the fluid pressure, $p_F$. In homogeneous isotropic linear elastic porous media, the equations for the quasi-static Biot model are:  

\begin{align}
\begin{split} \label{eq:2-field-eq}
\tred{- \div( 2\mu \e(\bs{u}) + \lambda \div \bs{u} \ul{\bs{I}} - \alpha p_F \ul{\bs{I}})} &= \bs{f}, \\
s_0 \dot{p}_F + \alpha \div \dot{\bs{u}} - \div (\kappa \nabla p_F) &= g, 
\end{split}
\end{align}
where the dots represent time derivatives, $\mu$ and $\lambda$ are the Lam\'e coefficients of elastic medium, $\e(\bs{u})$ is the symmetric gradient of $\bs{u}$, 
$\ul{\bs{I}}$ is the $n \times n$ identity matrix, $s_0 \geq 0 $ is the constrained specific storage coefficient, $\kappa>0$ is the hydraulic conductivity determined by the permeability of medium and the fluid viscosity, and $\alpha>0$ is the Biot--Willis constant which is close to 1. The system \eqref{eq:2-field-eq} can be posed  on a bounded domain in two and three space dimensions.
The given functions $\bs{f}$ and $g$ represent body force and source/sink of fluid, respectively. We will assume throughout the paper that the parameters $\mu$, $\lambda$, and $s_0$ are scalar functions on the domain, while in general $\kappa$ can be a symmetric positive definite matrix-valued function. To be a well-posed mathematical problem, the system  \eqref{eq:2-field-eq} needs appropriate boundary and initial conditions. A discussion of general boundary conditions for the Biot system will be given in Remark~\ref{rmk:general-bc} below. Furthermore,  a mathematical discussion of well-posedness of this model can be found in  \cite{MR1790411}.

Due to importance of Biot's model in applications, ranging from geoscience to medicine, finite element methods for the model have been studied by many researchers.  For example, various primal methods are studied in \cite{NAG:NAG1610080304,MR603175,NAG:NAG1610080106}, mixed methods in \cite{MR1156589,MR1257948,MR1393902}, Galerkin least square methods in  \cite{MR2177147}, discontinuous Galerkin methods in \cite{MR3047799}, and combinations of different methods in \cite{MR2327964,MR2327966,MR2461315,NUM:NUM21775}, but this list is by no means complete.

It is important to construct  numerical methods which are robust with respect to  variation of model parameters since this  variation in many practical problems is quite large.
For example, relevant parameters in the soft tissue of the central nervous system
are Young's modulus of $1-60$ kPa, Poisson ratio from 0.3 to almost 0.5 (0.499 in \cite{ben2012finite}), and   
the permeability is $10^{-14}-10^{-16} m^2$~\cite{smith2007interstitial,stoverud2014relation}.  
In geophysics, Young's modulus is typically in the order of GPa, 
Poisson ratio $0.1-0.3$, while the permeability
may vary from approximately $10^{-9}$ to $10^{-21} m^2$~\cite{coussy2004poromechanics,wang2000theory}. Relations of Young's modulus $E$, Poisson ratio $\nu$ and the two elastic moduli $\mu$, $\lambda$ are 
$\mu = E/2(1+\nu)$ and $\lambda = E \nu /(1 + \nu)(1 - 2 \nu)$.  Consequently, $\mu$ and $\lambda$ are in the ranges of $300-500$ MPa and $100-500$ MPa, respectively, in geoscience applications, whereas
corresponding numbers are $\mu$  and $\lambda$ in the ranges $300-2000$ Pa and $500-10^6$ Pa in neurological applications. 

However, in the present paper we will not limit ourselves to the study of robustness with respect to model parameters of the finite element discretization of Biot's model.
In fact, our main concern is to be able to construct preconditioners for the discrete systems which are well behaved both with respect to variations of the model parameters and the refinement of the discretization. When large discrete systems are solved by iterative methods, the convergence rate depends heavily on the construction of suitable preconditioners. Such 
preconditioners for finite element discretizations of  Biot's model have been studied by many authors, cf. for example  \cite{MR3148142,NAG:NAG164,MR2922283,NME:NME500,MR3249373}. Recently, there is also an emerging interest for preconditioners which are robust with respect to model parameters~\cite{MR3249373,MR3395132}. However, robustness with respect to all model parameters remains challenging. 
In particular, we will derive preconditioners that are robust as the medium approaches the incompressibility limit
while the permeability is low. In our experience this represents the most difficult case, and it is 
also the case that occurs in many biomechanical applications.  

The purpose of this paper is to develop a stable finite element method for Biot's model, and a corresponding  preconditioner for the associated discrete systems,  such that the preconditioned systems have condition numbers which are robust with respect to variations of model parameters. More precisely, we aim to have a preconditioned system which is robust for small $\kappa$, small time-steps, large $\lambda$, large $\mu$, and mesh refinements. In order to obtain such a parameter-robust preconditioner we employ the operator preconditioning framework of \cite{MR2769031}. It turns out that typical formulations of Biot's model are not appropriate to apply the framework, so we develop a new three-field formulation of Biot's model and propose a parameter-robust block diagonal preconditioner for it. 

The present paper is organized as follows. In Section 2, we introduce some notation and conventions that will be used throughout the paper. Furthermore, we  briefly discuss the preconditioning framework of \cite{MR2769031} based on parameter-robust stability of the continuous problems, and illustrate this with some numerical examples based on simplified models which can be seen as subsystems of the Biot system. In Section 3 we explain some difficulties  related to more common formulations of the Biot system, and 
as a consequence we motivate a new three-field formulation. 
The discussion of finite element discretizations 
based on this three-field formulation is given in Section 4, and the stability  results are used to motivate the construction of 
parameter robust preconditioners.
The implementation of a special operator related to one of the blocks a block diagonal preconditioner is discussed in Section 5. Finally, in Section 6 we present some numerical experiments which illustrate our theoretical results.

\section{Preliminaries}
The system \eqref{eq:2-field-eq} can in principle be studied on rather general domains $\Omega$ in two or three dimensions. However, our main goal is to study finite element approximations of this system, and therefore we will assume throughout this paper that $\Omega$ is a bounded \tblu{polyhedral} domain in $\R^n$, with $n=2$ or $3$. We will use $H^k = H^k(\Omega)$ to denote the Sobolev space of functions on $\Omega$ with $k$ derivatives in $L^2$ and the corresponding norm is denoted by $\|\cdot\|_k$. Further, let $H^k_0$ be the closure of $C^\infty_0(\Omega)$ in $H^k$ with dual space denoted by $H^{-k}$ and $(\cdot, \cdot)$ \tblu{denote} the $L^2$ inner product of scalar, vector, and matrix valued functions as well as the duality paring between $H^k_0$ and $H^{-k}$. The space $L_0^2$ is the space of $L^2$ functions with mean value zero. Boldface symbols are used to denote vector valued functions or spaces, and symbols of boldface with underline are used to denote matrix valued functions. 

Throughout this paper we use $A \lesssim B$ to denote the inequality $A \leq CB$ with a generic constant $C>0$ which is independent of the discretization parameters and the model parameters,
and $A \sim B$ will stand for $A \lesssim B$ and $B \lesssim A$. If needed, we will use $C$ to denote generic positive constants in inequalities. For a scalar valued function $g$, $\nabla g$ is a (column) vector valued function. For a matrix valued function $\bs{\ul{g}}$, $\div$ is understood as a row-wise divergence which results in the vector valued function $\div \bs{\ul{g}}$. Adopting these conventions, the equations \eqref{eq:2-field-eq} are well-defined.

\subsection{Preconditioning of parameter-dependent systems} \label{subsec:framework}
Let us briefly review the abstract framework of parameter-robust preconditioning in \cite{MR2769031}. Let $X$ be a separable, real Hilbert space with inner product $\langle \cdot, \cdot \rangle_X$ and the associated norm $\| \cdot \|_X$. For two Hilbert spaces $X$ and $Y$, $\mathcal{L}(X, Y)$ is the Hilbert space of bounded linear maps from $X$ to $Y$. Let us denote the dual space of $X$ by $X^*$ and the duality pairing of $X$ and $X^*$ by $\langle \cdot, \cdot \rangle$. 
Suppose that $\mathcal{A} \in \mathcal{L}(X, X^*)$ is invertible and also symmetric in the sense that 
\begin{align*}
\langle \mathcal{A} x, y \rangle = \langle x, \mathcal{A} y \rangle, \qquad x, y \in X.
\end{align*}
For given $f \in X^*$ consider a problem finding $x \in X$ such that 
\begin{align} \label{eq:Axf}
 \mathcal{A} x = f. 
\end{align}
Its preconditioned problem with a symmetric isomorphism $\mathcal{B} \in \mathcal{L}(X^*, X)$ is 
\begin{align*}
\mathcal{BA} x = \mathcal{B} f. 
\end{align*}
The convergence rate of a Krylov space method for this problem can be bounded by the condition number, $K(\mathcal{BA})$, given by 
\algns{
K(\mathcal{BA}) := 
\| \mathcal{BA} \|_{\mathcal{L}(X, X)} \| (\mathcal{BA})^{-1} \|_{\mathcal{L}(X, X)} .
}

Parameter-dependent problems are handled in this framework as follows. Let $\ve$ denote a collection of parameters, and $\mathcal{A}_{\ve}$ the parameter-dependent coefficient operator. 
A systematic way to construct an $\ve$-robust  preconditioner $\mathcal{B}_{\ve}$, as proposed in \cite{MR2769031}, is to consider
the mapping property of $\mathcal{A}_{\ve}$ in $\ve$-dependent Hilbert spaces $X_{\ve}$ and $X^*_{\ve}$.
The key property is to choose the spaces $X_{\ve}$ and $X^*_{\ve}$ such that 
$\mathcal{A_\ve}$ is a map from $X_{\ve}$ to $X^*_{\ve}$, and  with corresponding operator norms
$\| \mathcal{A}_{\ve} \|_{\mathcal{L}(X_{\ve}, X_{\ve}^*)} \mbox{ and } \| \mathcal{A}_{\ve}^{-1} \|_{\mathcal{L}(X^*_{\ve}, X_{\ve})} $
bounded independently of $\ve$. In this case the preferred preconditioner, $\mathcal{B_\ve}$, 
is a map from $X_{\ve}^*$ to $X_{\ve}$  with the property that 
$\| \mathcal{B}_{\ve} \|_{\mathcal{L}(X^*_{\ve}, X_{\ve})} \mbox{ and } \| \mathcal{B}_{\ve}^{-1} \|_{\mathcal{L}(X_{\ve}, X_{\ve}^*)}$
are bounded independently of $\ve$. If such an operator $\mathcal{B}_{\ve}$ is identified, then the condition number 
$K(\mathcal{B_\ve} \mathcal{A_\ve})$ is bounded independently of $\ve$, since both  
$\mathcal{B_{\ve}A_{\ve}} $ and $(\mathcal{B_{\ve}A_{\ve}})^{-1}$ are operators on  $\mathcal{L}(X_{\ve}, X_{\ve})$, with corresponding operator norms bounded independently of  
$\ve$. 

As we will illustrate below the discussion outlined above can often most easily be done in the continuous setting. On the other hand, in a computational  setting we need  preconditioners for the corresponding 
discrete  problems. In fact, if we utilize a finite element discretization which is uniformly stable with respect to the parameters, then the structure of the preconditioners in the discrete case will be the natural 
discrete analogs of the preconditioners in the continuous case. However, the preconditioners derived by the procedure above will often require exact inverses of operators which cannot be inverted cheaply. Therefore, 
in order to obtain effective  robust preconditioners in the discrete case, we also have to replace these exact inverses by 
related equivalent operators, often obtained by common procedures such as multilevel methods or domain decomposition methods. We refer to \cite{MR2769031} and the examples below for more details.

A challenge of the Biot system is the dependency of several different and independent parameters like 
the Lam\'{e} elastic parameters  as well as parameters related to porous flow such as permeability and the Biot-Willis constant.  
The aim is to achieve robustness with respect to all model parameters, as well as the resolution of the discretization.   
To motivate this discussion we start by considering two simplified examples related to the Biot equations. 
The first example illustrates the case where the permeability tends to zero,  while the Lam\'{e} parameters are of unit scale. 
In the second example we consider the case where the elastic material tends towards the  incompressible limit. 
Both examples are special cases of the Biot system.

\begin{eg} \label{thm:eg1} \normalfont
Consider a system of equations
\begin{align}
\begin{split} \label{eq:ex1-eq}
- \lap \bs{u} + \nabla p &= \bs{f} , \\
- \div \bs{u}  + \div(\kappa \nabla p) &= g ,
\end{split}
\end{align}
with unknowns $\bs{u} : \Omega \ra \R^n$ and $p : \Omega \ra \R$. As boundary conditions  we use homogeneous Dirichlet 
condition for $\bs{u}$, i.e.,  $\bs{u}|_{\pd \Omega} = 0$, while we use homogeneous Neumann condition for $p$. The  parameter $\kappa >0$ is taken to be a constant in this example. A variational formulation of this problem  is to find $(\bs{u}, p) \in \bs{H}_0^1 \times H^1\cap L_0^2$ satisfying 
\begin{align*}
(\nabla \bs{u}, \nabla \bs{v}) - (p, \div \bs{v}) &= (\bs{f} , \bs{v}), & & \forall \bs{v} \in \bs{H}_0^1, \\
-(\div \bs{u}, q)  - (\kappa \nabla p, \nabla q) &= (g, q), & & \forall q \in H^1 \cap L_0^2.
\end{align*}
This system has a form of \eqref{eq:Axf} with $X = \bs{H}_0^1 \times H^1\cap L_0^2$ and 
\[
\mathcal{A} = 
\left(
\begin{array}{cc}
- \Delta & \nabla \\  
-\div &   \div (\kappa\nabla)   
\end{array}
\right) .
\]
If we define $X_{\kappa}$ as the Hilbert space $\bs{H}_0^1 \times H^1\cap L_0^2$ with $\kappa$-dependent norm 
given by 
\begin{align*}
\| (\bs{u}, p) \|_{X_{\kappa}}^2 = \| \bs{u} \|_1^2 + \| p \|_0^2 + \kappa \| \nabla p \|_0^2,  
\end{align*}
then one can check that $\mathcal{A} : X_{\kappa} \ra X_{\kappa}^*$ is an isomorphism
with corresponding operator norms of $\mathcal{A}$ and $\mathcal{A}^{-1}$ bounded independently of 
$\kappa$. Here the norm on $X_{\kappa}^* \supset \bs{L}^2 \times L^2$ is defined from the norm on 
$X_{\kappa}$ by extending the $L^2$ inner product to a duality pairing.
To define a robust  preconditioner we need to identify an isomorphism $\mathcal{B} : X_{\kappa}^* \ra X_{\kappa}$ with corresponding operator norms 
of $\mathcal{B}$ and $\mathcal{B}^{-1}$
bounded independently of $\kappa$.
A natural choice is a block-diagonal operator  of the form 
\begin{align}  \label{eq:ex1-precond}
\mathcal{B} = 
\left(
\begin{array}{cc}
-\lap^{-1} & 0 \\
0 & \left(  I - \kappa \lap \right)^{-1}
\end{array} 
\right) .
\end{align}
However, the operator $\mathcal{B}$ only indicates the structure of the desired preconditioner of the discrete system.
We will discretize this problem on the unit square in $\R^2$ by the lowest order Taylor-Hood element
with respect to a mesh of uniform squares. This method is uniformly stable with respect to $\kappa$ in the proper norms.
Furthermore, to obtain an effective preconditioner the exact inverses in the definition of $\mathcal{B}$
are replaced by corresponding algebraic multigrid preconditioners for the operators $-\lap$ and $I - \kappa \lap$,  implemented in the software library Hypre \cite{Falgout2002} with default settings. 
The preconditioned system is implemented using cbc.block \cite{mardal2012block} and FEniCS \cite{fenicsbook}. 
The eigenvalue estimates are obtained by the conjugate gradient method of normal equation of the system with convergence criterion $10^{-16}$, and we refer to \cite{mardal2012block} for more details. 
The same setup is used throughout the paper.

We present numerical results in Table~\ref{ex1-exp}. The numbers of iterations and condition numbers increase as the value of $\kappa$ decreases but they are asymptotically stable and are still bounded in the limit case $\kappa=0$, which is the Stokes equation. We remark that the zero eigenvalue of the system associated with $H^1 \cap L_0^2$ is ignored in the computation of condition numbers.
%They increase for mesh refinements but the 
\end{eg}

\begin{table}[ht]  % file 'ex1AMG'
\caption{\small{Number of iterations of MinRes solver of system \eqref{eq:ex1-eq} with algebraic multigrid (AMG) preconditioner of the structure \eqref{eq:ex1-precond}. Estimates of the condition numbers of the preconditioned system are given in parenthesis. ($\Omega$ = unit square, partitioned as bisections of $N \times N$ rectangles, convergence criterion\protect\footnotemark}
with relative residual of $10^{-6}$)}
\label{ex1-exp}
\centering
\begin{tabular}{>{\small}c >{\small}c | >{\small}c >{\small}c >{\small}c >{\small}c >{\small}c}
\hline 
%\multirow{2}{*}{elements}	
& & \multicolumn{5}{>{\small}c}{$N$} \\ 
&  & $16$ & $32$ & $64$ & $128$ & $256$ 	 \\  
\hline 
\multirow{8}{*}{$\kappa$} 
& $10^0$    & $13 \;(1.3) $  & $ 13 \;(1.2) $ & $ 14 \;(1.2) $  &  $ 14 \;(1.2) $ & $ 14 \;(1.2) $   \\ 
& $10^{-1}$ & $16 \;(1.7) $  & $16\;(1.8) $ & $16 \;(1.8) $  &  $16 \;(1.8) $ & $16 \;(1.8) $   \\ 
& $10^{-2}$ & $22 \;(3.1) $  & $24 \;(3.4) $ & $23 \;(3.6) $  &  $23 \;(3.8) $ & $24 \;(3.9) $   \\ 
& $10^{-3}$ & $30 \;(4.6) $  & $29 \;(4.9) $ & $29 \;(5.4) $  &  $30 \;(5.8) $ & $30 \;(6.1) $   \\ 
& $10^{-4}$ & $35 \;(6.1) $  & $36 \;(6.2) $ & $36 \;(6.4) $  &  $35 \;(6.6) $ & $34 \;(7.0) $   \\ 
& $10^{-5}$ & $38 \;(7.0) $  & $39 \;(7.8) $ & $41 \;(7.9) $  &  $39 \;(7.7) $ & $39 \;(7.7) $   \\ 
& $10^{-6}$ & $38 \;(7.1) $  & $40 \;(8.3) $ & $42 \;(9.2) $  &  $44 \;(9.4) $ & $43 \;(9.0) $   \\ 
& $0$ 	    & $38 \;(7.1) $  & $42 \;(8.4) $ & $44 \;(9.5) $  &  $47 \;(10.5) $ & $48 \;(11.2) $   \\ 
\hline 
\end{tabular}  
\end{table}
\footnotetext{ Convergence criterion is $(\mathcal{B} r_k, r_k) / (\mathcal{B} r_0, r_0) \le 10^{-6}$ where $\mathcal{B}$ is preconditioner and $r_k$ is the residual of $k$-th iteration. }

\begin{eg} \label{thm:eg2} \normalfont
Consider a system related to the Lam\'{e} problem in linear elasticity: Find $\bs{u} :\Omega \ra \R^n$, $p : \Omega \ra \R$ for 
\begin{align} \label{eq:ex2-eq}
\begin{split}
- \div \e (\bs{u}) - \nabla p &= \bs{f}, \\
\div \bs{u} - \frac{1}{\lambda} p &= g, 
\end{split}
\end{align}
with $\bs{u}|_{\pd \Omega} = 0$, where $1 \leq \lambda < +\infty$ is a positive constant and $\e (\bs{u})$ is the symmetric gradient of $\bs{u}$. 
\tgreen{When $g = 0$, this is the Lam\'e problem and $p = \lambda \div u$ is called ``solid pressure''.}
Its variational form is to find $\bs{u} \in \bs{H}_0^1$ and $p \in L^2$ such that 
\begin{align*}
(\e (\bs{u}), \e(\bs{v})) + (p, \div \bs{v}) & = (\bs{f} , \bs{v}), & & \forall \bs{v} \in \bs{H}_0^1, \\
(\div \bs{u}, q) - \frac{1}{\lambda} (p, q) &= g, & & \forall q \in L^2 . 
\end{align*}
This is a saddle point problem with a stabilizing term $-(1/\lambda) (p,q)$, and the stabilization becomes weaker as $\lambda$ becomes larger. Let $X = \bs{H}_0^1 \times L^2$ be the Hilbert space with standard norm. In the limit when $\lambda = +\infty$ the system is not stable in these norms.
This is due to the fact that 
 $\div \bs{H}_0^1 \subsetneq L^2$ and as a consequence the Brezzi condition for stability of saddle point problem is not fulfilled \cite{BFBook}. More precisely, $\div \bs{H}_0^1$ can control only the $L^2$ norm of the mean-value zero part of $p$, and the stabilizing term is needed to control the mean-value part of $p$. Since the stabilizing term is dependent on $\lambda$, we need a $\lambda$-dependent norm on $\bs{H}_0^1 \times L^2$ to have $\lambda$-independent stability of the system. 

Before we define an appropriate $\lambda$-dependent norm, we need some preliminaries. Let $P_m$ be the linear operator in $L^2$ such that 
\begin{align*} 
P_m \phi := \left( \frac{1}{|\Omega|} \int_{\Omega} \phi \,dx \right) \chi_{\Omega}, \qquad \forall \phi \in L^2,
\end{align*}
where $\chi_{\Omega}$ is the characteristic function on $\Omega$ and $|\Omega|$ is the Lebesgue measure of $\Omega$. Notice that $P_m \phi$ and $\phi - P_m \phi$ are the decomposition of $\phi$ into its {\it mean-value} part and {\it mean-value zero} part. For $q \in L^2$ we denote its mean-value and mean-value zero parts by 
\begin{align} \label{eq:decomp}
%\tilde{Q} = \tilde{Q}_c \oplus \tilde{Q}_0, \qquad \tilde{Q}_c := \R, \quad \tilde{Q}_0 := L^2(\Omega)/{\R}, \\
{q}_{m} := P_m {q}, \qquad {q}_{0} = \tblu{q} - q_{m}. 
%\tilde{q}_c \in \mathcal{P}_0(\Omega), \quad \tilde{q}_0 \in L^2(\Omega)/{\R} .
\end{align}
We now define a Hilbert space $X_{\lambda}$ by 
\[
\| (\bs{v}, q) \|_{X_{\lambda}}^2 = \| \tblu{\bs{v}} \|_1^2 + \left( \frac{1}{\lambda} \| q_m \|_0^2 + \| q_0 \|_0^2 \right) ,
\]
then the system is $\lambda$-independent stable in $X_{\lambda}$. We will not give  a detailed proof of stability here since it can be obtained by modifying the proof of Theorem~\ref{thm:3f-inf-sup} below. 
But the rough explanation is that the mean value of $p$ cannot be control by the inf-sup condition,
and therefore this part of the norm has to be weighted properly in balance with the stabilizing term $\lambda^{-1} \|p_m\|_0^2$,
while the rest of $p$ is controlled by the inf-sup condition.

We present numerical results for two different preconditioners. The Hilbert space $X$ and $X_{\lambda}$ lead to preconditioners of the forms 
\begin{align} \label{eq:ex2-precond}
 \mathcal{B}_1 = 
\left(
\begin{array}{cc}
-\lap^{-1} & 0 \\  
0 & I^{-1}
\end{array}
\right) , \qquad
 \mathcal{B}_2 = 
\left(
\begin{array}{cc}
-\lap^{-1} & 0 \\  
0 & \left(I_0 + \frac{1}{\lambda} I_m \right)^{-1}
\end{array}
\right) .
\end{align}
Here the appearance of the operator $I^{-1}$ calls for an explanation. In fact, the operator $I$ should not be thought of as 
the identity operator on the Hilbert space $L^2$, but rather as the Riesz map between 
this space and its dual. In particular, in the corresponding discrete setting the operator $I^{-1}$ is typically 
represented by the 
inverse of a mass matrix. In a similar manner, the operators $I_0$ and $I_m$ should be interpreted as 
maps of $L^2$ into the duals of $L_0^2$ and its complement. We refer to \cite[Section 6]{MR2769031} for more details.

We employ the lowest order Taylor-Hood discretization and test the efficiency of the preconditioners on different refinements of the unit square.  
The preconditioner $\mathcal{B}_1$ is implemented by replacing the exact inverse of $-\lap$ by an algebraic multigrid operator, and we use the Jacobi iteration to approximate $I^{-1}$. 
The construction of $\mathcal{B}_2$ is technical due to the second block, and we postpone the details of the construction to Section~\ref{sec:weighted-jacobi}. 

Numerical results for the preconditioners $\mathcal{B}_1$ and $\mathcal{B}_2$ are given in Tables~\ref{ex2-expa} and \ref{ex2-expb}. The results for the preconditioner with structure $\mathcal{B}_2$ is completely satisfactory.
Both the number of iterations and the condition numbers appear to be uniformly bounded with respect to $\lambda$ and 
mesh refinement. However, the results for the preconditioner $\mathcal{B}_1$ are different.
In this case the condition numbers appear to grow linearly with $\lambda$, while the number of iterations 
still seems to be bounded, even if they appear to be slightly less robust in this case.
So in the case of the preconditioner $\mathcal{B}_1$ the condition number will not lead to a sharp bound on the
number of iterations.
In fact, by comparing the operator $\mathcal{B}_1$ with the uniform preconditioner $\mathcal{B}_2$ one can 
argue that the operator  $\mathcal{B}_1 \mathcal{A}$ has one isolated eigenvalue   
that tends to $0$ as $\lambda$ increases, while the rest of the spectrum lie on an interval bounded independently of the mesh resolution and $\lambda$.
In cases where only few eigenvalues are outside a bounded spectrum, it is well-known that the Conjugate Gradient and the Minimum Residual methods are very efficient~\cite{axelsson1986rate, nielsen2013analysis} and therefore $\mathcal{B}_1$ is about as efficient as $\mathcal{B}_2$, but slightly less robust.

%\tred{The numbers of iteration when for two preconditioners $\mathcal{B}_1$ and $\mathcal{B}_2$ are therefore similar as $N$ is large.}
%there is only one eigenvalue that approaches zero as $\lambda\rightarrow\infty$ when using $\mathcal{B}_1$. 
% \[ \mathcal{A} = 
% \left(
% \begin{array}{cc}
% - \div \e & \nabla \\  
% -\div & -\frac{1}{\lambda} I
% \end{array}
% \right)
% \]
% Note that $\div \e$ and $\lap$ are spectrally equivalent on $\bs{H}_0^1$ due to Korn's inequality. Thus an operator $\mathcal{B}$ of the form
% is a preconditioner of $\mathcal{A}$ which is robust in $\lambda$. See Table~\ref{ex2-exp} for numerical results.
\end{eg}

\begin{table}[ht]  % file 'ex2aAMG'
\caption{\small{Number of iterations of MinRes solver of system \eqref{eq:ex2-eq} with preconditioner of the form $\mathcal{B}_1$ in \eqref{eq:ex2-precond}. Estimates of the condition numbers of the preconditioned system are given in parenthesis. ($\Omega$ = unit square, partitioned as bisections of $N \times N$ rectangles, convergence criterion with relative residual of $10^{-6}$) }} \label{ex2-expa}
\centering
\begin{tabular}{>{\small}c >{\small}c | >{\small}c >{\small}c >{\small}c >{\small}c >{\small}c}
\hline 
%\multirow{2}{*}{elements}	
& & \multicolumn{5}{>{\small}c}{$N$} \\ %\hline %& \multirow{2}{*}{order } & \multirow{2}{*}{$n$ } \\ 
&  & $16$ & $32$ & $64$ & $128$ & $256$ 	 \\  
\hline 
\multirow{7}{*}{$\lambda$} 
& $10^0$ & $29 \;(3.5)$  	& $29 \;(3.6)$ 	& $29 \;(3.6)$  &  $29\;(3.6)$ & $29\;(3.6)$   \\ 
& $10^1$ & $40 \;(10.8)$  	& $41\; (10.8)$ & $38 \;(10.9)$  &  $38\;(10.9)$ & $36\;(10.9)$   \\ 
& $10^2$ & $53 \;(95.6)$  	& $59\; (96.2)$ & $54 \;(96.7)$  &  $53\;(96.9)$ & $52\;(97.0)$   \\ 
& $10^3$ & $60 \;(958)$  	& $62 \;(964)$ 	& $62 \;(969)$  &  $61\;(972)$ & $60\;(973)$   \\ 
& $10^4$ & $66 \;(9589)$  	& $69\; (9649)$ & $44 \;(9697)$  &  $43\;(9724)$ & $42\;(9734)$   \\ 
& $10^5$ & $44\;(95892)$  	& $44\;(96501)$ & $44 \;(96981)$  &  $43\;(97248)$ & $42\;(97344)$   \\ 
& $10^6$ & $44\;(958942)$  	& $44\;(964970)$& $44\;(969940)$  &  $43\;(972566)$ & $41\;(973587)$   \\ 
\hline 
\end{tabular}                          
\end{table}

\begin{table}[ht]  % file 'ex2bAMG'
\caption{\small{Number of iterations of MinRes solver of system \eqref{eq:ex2-eq} with preconditioner of the form $\mathcal{B}_2$ in \eqref{eq:ex2-precond}. Estimates of the condition numbers of the preconditioned system are given in parenthesis. ($\Omega$ = unit square, partitioned as bisections of $N \times N$ rectangles, convergence criterion with relative residual of $10^{-6}$) }} \label{ex2-expb}
\centering
\begin{tabular}{>{\small}c >{\small}c | >{\small}c >{\small}c >{\small}c >{\small}c >{\small}c}
\hline 
%\multirow{2}{*}{elements}	
& & \multicolumn{5}{>{\small}c}{$N$} \\ %\hline %& \multirow{2}{*}{order } & \multirow{2}{*}{$n$ } \\ 
&  & $16$ & $32$ & $64$ & $128$ & $256$ 	 \\  
\hline 
\multirow{7}{*}{$\lambda$} 
& $10^0$ & $29\;(3.5) $  & $29\;(3.6) $ & $29\;(3.6) $  &  $29\;(3.6) $ & $29\;(3.6) $   \\ 
& $10^1$ & $41\;(11.6)$  & $41\;(11.6)$ & $41\;(11.6)$  &  $38\;(11.6)$ & $38\;(11.6)$   \\ 
& $10^2$ & $46\;(18.4)$  & $45\;(18.4)$ & $44\;(18.4)$  &  $44\;(18.4)$ & $44\;(18.4)$   \\ 
& $10^3$ & $46\;(19.6)$  & $46\;(19.6)$ & $45\;(19.6)$  &  $45\;(19.6)$ & $44\;(19.6)$   \\ 
& $10^4$ & $45\;(19.7)$  & $46\;(19.7)$ & $45\;(19.7)$  &  $45\;(19.7)$ & $44\;(19.7)$   \\ 
& $10^5$ & $45\;(19.7)$  & $46\;(19.7)$ & $46\;(19.7)$  &  $44\;(19.7)$ & $44\;(19.7)$   \\ 
& $10^6$ & $46\;(19.7)$  & $46\;(19.7)$ & $43\;(19.7)$  &  $44\;(19.7)$ & $44\;(19.7)$   \\ 
\hline 
\end{tabular}                          
\end{table}

\subsection{Parameter rescaling of the Biot system} 
The systems in the above two examples have only one parameter, so it is relatively easy to find function spaces and parameter-dependent norms \tblu{such} that the aforementioned preconditioner construction is applicable. However, the Biot system has several parameters of different ranges, so it is not easy to find appropriate function spaces and their parameter-dependent norms. In the rest of this section, we will rescale parameters of the Biot system and reduce it to a problem with three parameters. This procedure will not only simplify the problem but also clarify intrinsic parameters of the system. We emphasize that $\mu$, $\lambda$, $s_0$, and $\kappa$ are allowed to be functions on the domain, while $\alpha$ is assumed to be a constant.

Recall that when we solve a time-dependent problem numerically, we discretize the problem in time and solve a static problem at each time step, so preconditioning of time-dependent problem is reduced to preconditioning of static problem at each time step. If we consider an implicit time discretization (e.g., the backward Euler method) \tblu{applied} to \eqref{eq:2-field-eq} with time-step size $\delta^2$ ($0 < \delta \leq 1$), and multiply the second equation with $-\delta^2$, then we obtain  a static problem 
\begin{align*}
\tred{- \div ( 2\mu \e(\bs{u}) + \lambda \div \bs{u} \ul{\bs{I}} - \alpha p_F \ul{\bs{I}} )} &= \bs{f}, \\
- s_0 p_F - \alpha \div \bs{u} + \delta^2 \div (\kappa \nabla p_F) &= \tilde{g}
\end{align*}
with some right hand side $\tilde{g}$. 

To reduce this system further, we recall the physical derivation of $s_0$. The storage coefficient $s_0$ is the increase of the amount of fluid for the unit increase of fluid pressure, when volumetric strain is kept as constant. More precisely, $s_0 = \phi c_F + (1-\phi) c_S$ where $\phi$ is the porosity of solid, and $c_F \geq 0$, $c_S \geq 0$ are compressibilities of the fluid and solid. For derivation of these equations from physical modeling, we refer to standard porous media references, for instance, \cite{anandarajah2010computational}.
The parameter $c_S$ and other parameters $\mu$ and $\lambda$ are related so that $c_S \sim 1/(2\mu / n + \lambda)$ holds \tblu{with $n$, the spatial dimension of $\Omega$}. In many practical applications, $\mu \lesssim \lambda$ and $c_F \approx 0$ hold, so we have $s_0 \sim 1/ \lambda$. Furthermore, $\alpha$ is a constant close to 1, so we will assume that  $s_0$ scales like $\alpha^2 / \lambda$. Therefore, to reduce the number of parameters in our system we will simply let 
$s_0 = \alpha^2/\lambda$ for the rest of the discussion in this paper. We emphasize that the equality is not essential.
Our analysis below can easily be adopted to the situation where 
$s_0$ scales like $\alpha^2 / \lambda$. However, this rather artificial expression of $s_0$ is useful when we normalize the system later. In addition, $\delta^2 \kappa$ can be regarded as one small parameter because the hydraulic conductivity $\kappa$ is small in general.
As a consequence, introducing $\tilde{\kappa} = \delta^2 \kappa$, we have a simplified system 
\begin{align*}
%\label{eq:red-eq1} 
\tred{- \div (2 \mu \e(\bs{u}) + \lambda \div \bs{u} \ul{\bs{I}} - \alpha p_F \ul{\bs{I}}) } &= \bs{f}, \\
%\label{eq:red-eq2} 
- \frac{\alpha^2}{\lambda} p_F - \alpha \div \bs{u} + \div (\tilde{\kappa} \nabla p_F) &= \tilde{g}, 
\end{align*}
with $1 \leq \lambda < + \infty$, small $\tilde{\kappa}$, and $\alpha \sim 1$. However, in practical applications, $\mu$ can be much larger than $\alpha$, so we rescale the above equations to include this factor. 
To do so we assume $\mu$ is a spatial function with a uniform  scale, i.e., there is a constant $\bar \mu$ such that 
$\mu/\bar \mu \sim 1$. 
Let 
\begin{align*}
\lambda' = \frac{\lambda}{2 \bar \mu}, \quad \mu' = \frac{\mu}{\bar \mu}, \quad\alpha' = \frac{\alpha}{2 \bar \mu}, \quad \kappa' = \frac{\tilde{\kappa}}{2 \bar \mu}, \quad \bs{f}' = \frac{\bs{f}}{2\bar \mu}, \quad g' = \frac{\tilde{g}}{2 \bar \mu}. 
\end{align*}
Dividing the above two equations by $2 \bar \mu$, the final simplified equations are 
\begin{align}
\begin{split} \label{eq:red-eq}
\tred{ - \div (\mu' \e(\bs{u}) - \lambda' \div \bs{u} \ul{\bs{I}} - \alpha' p_F \ul{\bs{I}}) } &= \bs{f}', \\
- \tgreen{\frac{(\alpha')^2}{\lambda'}} p_F - \alpha' \div \bs{u} + \div (\kappa' \nabla p_F) &= g', 
\end{split}
\end{align}
where $\alpha'$ is a small positive constant, $\lambda'$ and $\mu'$ are positive scalar functions such that $\lambda'$ is 
bounded from below and $\mu' \sim 1$, while $\kappa'$ is a positive definite matrix valued function with eigenvalues bounded from above.

In summary, we have reduced the original Biot system to a system with three intrinsic parameters $\lambda'$, $\alpha'$, and
$\kappa'$ which all may be unbounded. More precisely, the positive constant $\alpha'$ may be arbitrarily small,
the scalar function $\lambda'$ may be arbitrarily large, while the positive eigenvalues of $\kappa'$ may be close to zero.
Since $\mu'$ is bounded from above and below this parameter has no essential effect on the properties of the system. Therefore, to reduce the number of parameters of the system, we take $\mu' =1$ in the discussion below.

For the biomedical applications discussed in the introduction in units of Pascal, gram, milimeter and second,  
$\mu$ was in the order of $1-60$ kPa which makes  
$\lambda'$, $\alpha'$, $\kappa'$ in the ranges of $0.25-500$, $10^{-3}-10^{-5}$, $10^{-8}-10^{-12}$, respectively.        
In geoscience, a representative $\mu$ is 10 GPa and units for pressure, viscosity and permeability  are pounds per square inch (psi), centi Poise (cP), mili Darcy (mD).   
Using these units, $\lambda'$, $\alpha'$, $\kappa'$ are in the ranges of $0.25-3.5$, $10^{-10}$, $10^{-9}-10^{-13}$, respectively.

\section{Parameter-robust stability of the continuous problems}
For the rest of this paper we will discuss a system of the form \eqref{eq:red-eq}, where the parameter $\mu' = 1$. By omitting the primes on the parameters we obtain a system of the form
\begin{align}
\begin{split} \label{eq:scaled}
 - \div ( \e(\bs{u}) - \lambda \div \bs{u} \ul{\bs{I}} - \alpha p_F \ul{\bs{I}})  &= \bs{f}, \\
- \frac{\alpha^2}{\lambda} p_F - \alpha \div \bs{u} + \div (\kappa \nabla p_F) &= g. 
\end{split}
\end{align}
Throughout the rest of the paper we will assume that the parameters $\lambda$, $\alpha$, and $\kappa$ satisfy 
\begin{equation} \label{eq:param-range}
1 \leq \lambda < + \infty, \qquad  0 < \alpha \leq 1,    \qquad    0 <  \kappa \leq 1,
\end{equation}
where the assumption on the matrix valued function $\kappa$ has the interpretation that $\kappa$ is uniformly positive definite, and  with all eigenvalues bounded above by one.

\subsection{\tgreen{Difficulties} in typical formulations} 

In the discussion of the two examples in Section 2 we saw that simplified versions of the Biot system  were efficiently handled with 
straightforward extensions of standard preconditioners for Stokes problem. In particular, in Example   
\ref{thm:eg1}, the presence of a small permeability was handled by extending a standard Stokes preconditioner in the canonical way with an
operator of the form  
$\nabla\cdot(\kappa\nabla)$ applied to the pressure. Furthermore, in Example \ref{thm:eg2}, the use of ``solid pressure'' gave  a stable formulation and an efficient preconditioner even in 
the limit of an incompressible material. We will now demonstrate that extending these two approaches to the Biot system is not straightforward.

For simplicity we will first consider the system  with  homogeneous Dirichlet boundary conditions of the form  $\bs{u} = 0$ and $p_F = 0$ on $\pd \Omega$.
A discussion of other possible boundary conditions is given in Remark~\ref{rmk:general-bc} below. Throughout this subsection we will
make the simplifying assumption that $\lambda$ is a constant, and we will illustrate that parameter-robust preconditioning is difficult even in
that case. For a variational formulation of \eqref{eq:scaled} with Dirichlet boundary conditions, we will use the function spaces ${\bs H}_0^1$ and $H_0^1$ for the unknowns $\bs{u}$ and $p_F$, respectively, and obtain
\algns{  
(\e(\bs{u}), \e(\bs{v})) + \lambda (\div \bs{u}, \div \bs{v}) - \alpha (p_F , \div \bs{v}) &= (\bs{f}, \bs{v}), & & \forall \bs{v} \in \bs{H}_0^1, \\
- \alpha (\div \bs{u}, q_F) - \frac{\alpha^2}{\lambda} (p_F, q_F) - (\kappa \nabla p_F, \nabla q_F) &= (g, q_F), & & \forall q_F \in H_0^1 .
}
In matrix form, the system is 
\begin{align} \label{eq:2f-A}
\mathcal{A} \begin{pmatrix}
\bs{u} \\
p_F
\end{pmatrix} := 
\begin{pmatrix}
-\div (\e + \lambda \div \ul{\bs{I}}) & \alpha \nabla \\
-\alpha \div & - \frac{\alpha^2}{\lambda} + \div (\kappa \nabla) 
\end{pmatrix}
\begin{pmatrix}
\bs{u} \\
p_F
\end{pmatrix} = 
\begin{pmatrix}
\bs{f} \\
g 
\end{pmatrix} .
\end{align}
This system has a perturbed saddle point problem structure. Following the preconditioner construction framework in the previous section, we define Hilbert spaces $\bs{V}_1$ and $Q_{F,1}$ with parameter-dependent norms on $\bs{H}_0^1$ and $H_0^1$,
%In order to have a uniformly bounded norm of $\mathcal{A} : \bs{V} \times Q \ra \bs{V}^* \times Q^*$, we define norms of $\bs{V}$ and $Q$ by 
\begin{align*}
\| \bs{v} \|_{\bs{V}_1}^2 &:= \| \e(\bs{v}) \|_0^2 + \lambda \| \div \bs{v} \|_0^2, & & \bs{v} \in {\bs H}_0^1, \\ %\| p \|_{Q_0} := \lambda^{-\half} \| p \|_0, \\ 
\| q_F \|_{Q_{F,1}}^2 &:= \frac{\alpha^2}{\lambda} \| q_F \|_0^2 + (\kappa \nabla q_F, \nabla q_F), & & q_F \in H_0^1.
\end{align*}
Then we are able to show that $\mathcal{A} : \bs{V}_1 \times Q_{F,1} \ra \bs{V}_1^* \times Q_{F,1}^*$ in \eqref{eq:2f-A} is an isomorphism with upper and lower bounds uniform in $\lambda$, $\kappa$, and $\alpha$.
However, this formulation has a nontrivial difficulty to achieve parameter-robust preconditioner for its discrete counterpart. For example, if we consider a block-diagonal preconditioner as in examples in the previous section, we need a good preconditioner of $- \div \e - \lambda \grad \div$ in first block of the preconditioner. However, usually accepted preconditioners (e.g., algebraic multigrid preconditioner) do not perform well when $\lambda$ is large. This is observed in \cite{MR3249373} for the simplified McKenzie equations, and can be explained by the fact that it is hard to construct discretizations which are 
 uniformly stable with respect to $\lambda$.

This difficulty is similar to volumetric locking problem in linear elasticity \cite{MR1174468}, which arises when $\lambda$ is very large. Thus, we expect that resolutions of the locking problem in elasticity are useful to circumvent this preconditioning problem. 
%This is not only a matter of preconditioning but also the one of convergence of numerical solutions.
%To see it heuristically, note that $\div \bs{u}$ is close to zero when $\lambda$ is very large because $\lambda \div \bs{u}$ is a finite physical quantity. Hence the finite element space for $\bs{u}$ should be able to approximate both of $\bs{u}$ and $\div \bs{u} \approx 0$ simultaneously but standard $C^0$ finite elements do not fulfill this requirement unless polynomials degree is very high. 
There are two mathematically equivalent ways to avoid the locking problem: one is reduced integration technique \cite{Malkus-Hughes-1978} and the other is the mixed method (see, e.g., \cite{Boffi-stokes,MR1140646,MR0343661,MR1950614}). However, both of them have technical difficulties in parameter-robust preconditioning. Here we will discuss the difficulty 
with the mixed approach.
%
%If the reduced integration technique is used in the three dimensions, then the vector-valued cubic Lagrange finite element is needed for $\bs{u}$. This yields a system with large number of degrees of freedom, and its convergence rate is still suboptimal. In addition, it seems that there are only a few results on preconditioners for the system with reduced integration technique \cite{MR1724348}.
Motivated by Example \ref{thm:eg2} and the mixed finite element technique to avoid   
the locking problem in linear elasticity, it is tempting to employ the ``solid pressure'' $p_S := \tblu{-\lambda \div \bs{u}}$.  Recall that our purpose here is to show that this formulation is \tblu{{\it not}} appropriate for parameter-robust preconditioner construction, \tblu{so we show the lack of stability and bad numerical results only for one specific case $\alpha = 1$}. Introducing $p_S = \tblu{-\lambda \div \bs{u}}$ in the first equation of \eqref{eq:scaled}, with some algebraic manipulation, we have a three-field formulation 
\begin{align}
\notag -\div \e(\bs{u}) + \nabla p_S + \nabla p_F &= \bs{f}, \\
\label{eq:3f-eq} -\div \bs{u} - \frac{1}{\lambda} p_S   &= 0, \\
\notag \tblu{- \div {\bs{u}} } - \frac{1}{\lambda} p_F + \div (\kappa \nabla p_F) &= g.
\end{align}
Since $\bs{u} \in {\bs H}_0^1$, $\div \bs{u}$ is mean-value zero, and therefore $p_S = \tblu{-\lambda \div \bs{u}}$ is mean-value zero as well. This means that, $L_0^2$ is an appropriate function space for $p_S$ in variational formulation. 
Thus, a variational form of \eqref{eq:3f-eq} is to find $(\bs{u}, p_S, p_F) \in {\bs H}_0^1 \times L_0^2 \times H_0^1$ such that 
\begin{align*}
(\e(\bs{u}), \e(\bs{v})) - (p_S + p_F, \div \bs{v}) &= (\bs{f}, \bs{v}), & & \forall \bs{v} \in {\bs H}_0^1, \\
-(\div \bs{u}, q_S) - \frac{1}{\lambda} (p_S, q_S) &= 0, & & \forall q_S \in L_0^2, \\
-(\div \bs{u}, q_F) - \frac{1}{\lambda} ( p_F, q_F ) - (\kappa \nabla p_F, \nabla q_F) &=  (g, q_F), & & \forall q_F \in H_0^1. 
\end{align*}
In order to have parameter-robust stability, we need to find parameter-dependent norms of ${\bs H}_0^1 \times L_0^2 \times H_0^1$ such that the above system gives a linear isomorphism from the Hilbert space to its dual space, and norms of the linear isomorphism and its inverse are independent of the parameters. The bilinear forms in the system, 
\begin{align}
\label{eq:3f-bil1} &(\e(\bs{u}), \e(\bs{v})), & & \bs{u}, \bs{v} \in  {\bs H}_0^1 , \\
\label{eq:3f-bil2} &(\div \bs{v}, p_S), & & \bs{v} \in {\bs H}_0^1, p_S \in L_0^2, \\
\label{eq:3f-bil3} &(\div \bs{v}, p_F), & & \bs{v} \in {\bs H}_0^1, p_F \in H_0^1, \\
\notag &\frac{1}{\lambda} (p_S, q_S), & & p_S, q_S \in  L_0^2 , \\
\label{eq:3f-bil5} &\frac{1}{\lambda} (p_F, q_F) - ( \kappa \nabla p_F, \nabla q_F), & & p_F, q_F \in H_0^1 ,
\end{align}
have to be bounded for the parameter-dependent norms, so some necessary conditions of the norms will be given. The bilinear form \eqref{eq:3f-bil1} suggests $H^1$-norm for ${\bs H}_0^1$. To make \eqref{eq:3f-bil2} and \eqref{eq:3f-bil3} bounded, the chosen norms of $L_0^2$ and $H_0^1$ have to bound the standard $L^2$-norms of $p_S$ and $p_F$, respectively. Finally, \eqref{eq:3f-bil5} enforces the norm of $H_0^1$ to be an upper bound of $(\kappa \nabla p_F, \nabla p_F)^{1/2}$. Thus the smallest possible norms for ${\bs H}_0^1$, $L_0^2$, $H_0^1$ from these observations are 
\begin{align} \label{eq:3fps-norms}
\| \e(\bs{u}) \|_0, \qquad \| p_S \|_0, \qquad \left( \| p_F \|_0^2 + (\kappa \nabla p_F, \nabla p_F) \right)^{\half}
\end{align}
for $\bs{u} \in {\bs H}_0^1$, $p_S \in L_0^2$, $p_F \in H_0^1$. 

We use ${\bs V}_2$, $Q_S$, $Q_{F,2}$ to denote the Hilbert spaces on ${\bs H}_0^1$, $L_0^2$, $H_0^1$ with the above norms. It is easy to check that all bilinear forms are bounded with these norms uniformly in $\lambda$ and $\kappa$. In other word, the linear map from ${\bs V}_2 \times Q_S \times Q_{F,2}$ to ${\bs V}_2^* \times Q_S^* \times Q_{F,2}^*$, given by the above three-field formulation has a uniform bound independent of the parameters. Unfortunately, it does not seem to be the case for the inverse of the linear map. 
Although the system is a stabilized saddle point problem, the stabilization terms $- \lambda^{-1} (p_S, q_S)$ and $- \lambda^{-1} (p_F, q_F)$ are not enough to control the $L^2$ norms of $p_S$ and $p_F$ when $\lambda$ is very large. Therefore, we need to control the norms by $\bs{v} \in {\bs V}_2$ with inf-sup condition. In other words, we need  
\begin{align} \label{eq:missing-inf-sup}
\inf_{(q_S, q_F) \in Q_S \times Q_{F,2}} \sup_{\bs{v} \in {\bs V}_2} \frac{(\div \bs{v}, q_S + q_F)}{\| \bs{v} \|_1 (\| q_S \|_0 + \| q_F \|_0)} \ge \beta >0,
\end{align}
with a constant $\beta$ which is independent of the parameters. 
However, both $q_S$ and $q_F$ interact with $\div \bs{v}$ in the bilinear form $(\div \bs{v}, q_S + q_F)$, and it is difficult to control two independent terms with one object, $\div \bs{v}$. When $\kappa$ is not small the stabilization term $(\kappa \nabla p_F, \nabla p_F)$ can be used to control the $L^2$-norm of $p_F$. However, as we have seen in the model reduction in the previous section, the smallness of $\kappa$ is given not only by small hydraulic conductivity in the model, but also by a small time-step. Thus it is inevitable to assume that $\kappa$ is small when we solve static problems at each time step. 

\begin{eg} \normalfont
We present a computational result which provides numerical evidence for the above discussion. Suppose that $\Omega$ is the unit square in $\R^2$ and 
\begin{align*} 
\Gamma_d = \{ (x, y) \in \R^2 \;:\; (x,y) \in \pd \Omega, x < 1 \}. 
\end{align*}
For simplicity of implementation, we assume that $\bs{u}$ is vanishing on $\Gamma_d$, not on $\pd \Omega$, and therefore the appropriate function space for $p_S$ is $L^2$, since $\lambda \div \bs{u} \not \in L_0^2$. This is a reasonable assumption since robust preconditioners should cover problems with general boundary conditions. 
\begin{table}[ht]  %file  'biot_3f_ps'
\caption{\small{Number of iterations for different \tblu{$\lambda$ and $\kappa$} with the preconditioner of the form in \eqref{eq:3f-ps-precond}. ($\Omega$ = unit square, partitioned as bisections of $N \times N$ rectangles, convergence criterion with relative residual of $10^{-6}$)}} \label{3f-ps-exp}
\centering
\begin{tabular}{>{\small}c >{\small}c | >{\small}c >{\small}c >{\small}c >{\small}c >{\small}c >{\small}c >{\small}c}
\hline 
%\multirow{2}{*}{elements}	
& & \multicolumn{7}{>{\small}c}{$k$ \quad$(\kappa = 10^{k})$} \\ %\hline %& \multirow{2}{*}{order } & \multirow{2}{*}{$n$ } \\ 
$N$& $\lambda$ & $0$ & $-1$ & ${-2}$ & $-3$ & ${-4}$ & $-5$ & ${-6}$	 \\  
\hline \hline
\multirow{4}{*}{$32$} 
& $10^0$ & $21$  & $20$ & $21$  &  $ 22$& $ 24$ & $ 25$& $ 25$  \\ 
& $10^2$ & $45$  & $46$ & $49$  &  $ 88$& $115$ & $111$& $ 84$  \\ 
& $10^4$ & $46$  & $46$ & $50$  &  $102$& $247$ & $359$& $198$  \\ 
& $10^6$ & $46$  & $46$ & $53$  &  $102$& $251$ & $451$& $238$  \\ \hline \multirow{4}{*}{$64$} 
& $10^0$ & $20$  & $20$ & $20$  &  $21$& $ 23$ & $ 23$& $ 24$  \\ 
& $10^2$ & $44$  & $44$ & $49$  &  $85$& $107$ & $107$& $ 92$  \\ 
& $10^4$ & $44$  & $44$ & $48$  &  $98$& $242$ & $315$& $207$  \\ 
& $10^6$ & $44$  & $44$ & $50$  &  $98$& $244$ & $370$& $231$  \\ \hline \multirow{4}{*}{$128$} 
& $10^0$ & $19$ &$19$ & $19$ & $21$& $ 22$ & $ 23$& $ 23$  \\ 
& $10^2$ & $40$ &$40$ & $44$ & $82$& $ 97$ & $ 99$& $ 93$  \\ 
& $10^4$ & $42$ &$42$ & $47$ & $94$& $220$ & $241$& $181$  \\ 
& $10^6$ & $40$ &$41$ & $46$ & $94$& $228$ & $262$& $189$  \\ 
\hline 
\end{tabular}                           
\end{table}
For discretization we use elements inspired by the lowest order Taylor--Hood element, i.e., $(\bs{\mathcal{P}}_2, \mathcal{P}_1, \mathcal{P}_1)$ Lagrange finite elements for $(\bs{u}, p_S, p_F)$. In the experiment, block-diagonal preconditioner $\mathcal{B}$ based on the norms in \eqref{eq:3fps-norms}, i.e., 
\begin{align} \label{eq:3f-ps-precond}
\mathcal{B} = 
\begin{pmatrix} 
- \lap^{-1} & 0 & 0 \\
0 & I^{-1}& 0 \\
0 & 0 & \left(  I - \kappa \lap \right)^{-1}
\end{pmatrix} 
\end{align}
is implemented using FEniCS with Hypre. 
As above, the exact inverses are replaced by suitable algebraic multigrid operators.
A heuristic way to validate this choice of preconditioner is to consider a special case of \eqref{eq:3f-eq} with $\lambda=1$. Noting that this special case has the systems in Example~\ref{thm:eg1} with $\lambda=1$ and in Example~\ref{thm:eg2} as subsystems, the operator $\mathcal{B}$ in \eqref{eq:3f-ps-precond} can be viewed as a combination of the robust preconditioners in those examples.

Numbers of iterations for different constant parameter values are given in Table~\ref{3f-ps-exp}. When $\lambda$ is not too large the numbers of iterations are more or less robust with respect to  variation of $\kappa$. 
However, the number of iterations clearly increases with  increasing $\lambda$, and the increment is quite large when $\kappa$ is small. The growth in iterations as $\lambda$ increases is milder when $\kappa$ is larger, which supports  our heuristic analysis of ``partial remedy'', that $(\kappa \nabla p_F, \nabla p_F)$ may play a role of a stabilization term for the $L^2$ norm of $p_F$.  
\end{eg}

\subsection{A new three-field formulation} 
The discussion in the previous subsection suggests that we need a different formulation to obtain  parameter-robust preconditioners. 
We will present such a new formulation of the system \eqref{eq:scaled} here, where parameters $\lambda$, $\alpha$, $\kappa$ are assumed to be as specified in the beginning of this section. In particular, $\lambda$ and $\kappa$ are allowed to be functions of the spatial domain.

The main obstacle in the discussion above was  that the inf-sup condition \eqref{eq:missing-inf-sup} is not fulfilled. To circumvent this, we use a different system with unknowns $(\bs{u}, p_T, p_F)$, where \tblu{$p_T := -\lambda \div \bs{u} + \alpha p_F$}. With this new unknown $p_T$, which will be called total pressure, we can rewrite the equations \eqref{eq:scaled} as 
\begin{align}
\notag -\div \e(\bs{u}) + \nabla p_T &= \bs{f}, \\
\notag -\div \bs{u} - \lambda^{-1} (p_T - \alpha p_F)   &= 0, \\
\notag \lambda^{-1} (\alpha p_T - 2 \alpha^2 p_F) + \div (\kappa \nabla p_F) &= g.
\end{align}
The matrix form of this system is 
\begin{align} \label{eq:n3f-mat}
{\mathcal{A}} \begin{pmatrix}
\bs{u} \\
p_T \\
p_F
\end{pmatrix} := 
\begin{pmatrix}
-\div \e & \nabla & 0 \\
-\div  & -\lambda^{-1} & \alpha \lambda^{-1} \\
0 & \alpha \lambda^{-1} & - 2 \alpha^2 \lambda^{-1} + \div (\kappa \nabla) 
\end{pmatrix}
\begin{pmatrix}
\bs{u} \\
p_T \\
p_F
\end{pmatrix} = 
\begin{pmatrix}
\bs{f} \\
0  \\
-g
\end{pmatrix} .
\end{align}
With function spaces $\bs{H}_0^1$, $L^2$, $H_0^1$ for $\bs{u}$, $p_T$, $p_F$, respectively, 
corresponding to Dirichlet boundary conditions for $\bs{u}$ and  $p_F$,
we obtain a variational form %\dfn{Rv2-C10, Rv1-C8}
\begin{align}
\notag (\e(\bs{u}), \e(\bs{v})) - (\div \bs{v}, \tblu{p_T}) &= (\bs{f}, \bs{v}), & & \bs{v} \in \bs{H}_0^1, \\
\label{eq:n3f-var} - (\div \bs{u}, q_T) - (\lambda^{-1} p_T, q_T) + (\alpha \lambda^{-1} p_F, q_T) &= 0, & & q_T \in L^2, \\
\notag (\alpha \lambda^{-1} p_T, q_F) -  2 (\alpha^2 \lambda^{-1} p_F, q_F) - (\kappa \nabla p_F, \nabla q_F) &= (g, q_F), & & q_F \in H_0^1.
\end{align}
Recall \tblu{that we used the decomposition $p = p_m + p_0$ and the stabilization term in order to obtain the stability of the system in Example~\ref{thm:eg2}. We need a similar argument to show the stability of \eqref{eq:n3f-var} due to the same reason, $\div \bs{H}_0^1 \subsetneq L^2$. 
Denoting the mean-value zero part of $q_T$ by $q_{T,0}$ as in \eqref{eq:decomp},} we define norms by 
\begin{align} \label{eq:n3f-norms}
\| \e(\bs{u}) \|_0 , \quad \left( (\lambda^{-1} p_T,  p_{T}) + \| p_{T,0} \|_0^2 \right)^{\half}, \quad \left ( (\alpha^2 \lambda^{-1} p_F, p_F) + (\kappa \nabla p_F, \nabla p_F) \right)^{\half}
\end{align}
for $\bs{u} \in \bs{H}_0^1$, $p_T \in L^2$, $p_F \in H_0^1$. Let us denote these spaces with the norms in \eqref{eq:n3f-norms} by $\bs{V}$, $Q_T$, $Q_F$, and let $\mathcal{X} = \bs{V} \times Q_T \times Q_F$. 
Then it can be shown that all the bilinear forms in \eqref{eq:n3f-var}  are uniformly bounded in  $\lambda$, $\alpha$, $\kappa$. In other words, 
the operator ${\mathcal{A}}$ appearing in \eqref{eq:n3f-mat} is a bounded linear map from $\mathcal{X}$ to $\mathcal{X}^*$ and its norm is independent of the three parameters. Here the norm on the space $\mathcal{X}^*$ is derived from the norm on $\mathcal{X}$
exactly as we explained in Example~\ref{thm:eg1}.

The following theorem implies that ${\mathcal{A}}$ is invertible and that the inverse is a map from $\mathcal{X}^*$ to $\mathcal{X}$
with operator norm independent of the three  parameters. 

\begin{theorem} \label{thm:3f-inf-sup}
For the system \eqref{eq:n3f-var} there exists a constant ${\beta} >0$, independent of $\lambda$, $\alpha$, $\kappa$ satisfying \eqref{eq:param-range}, such that the following  inf-sup condition holds: 
\begin{align*}
\inf_{ (\bs{u}, p_T, p_F) \in \mathcal{X} } 
\sup_{(\bs{v}, q_T, q_F) \in \mathcal{X}} 
\frac{( {\mathcal{A}} (\bs{u}, p_T, p_F), (\bs{v}, q_T, q_F))_{(\mathcal{X}^*, \mathcal{X})} } 
{\| (\bs{u}, p_T, p_F) \|_{\mathcal{X}} \| (\bs{v}, q_T, q_F) \|_{\mathcal{X}} }  \geq {\beta} .
\end{align*}
\end{theorem}
\begin{proof}
To prove the inf-sup condition, we will use a standard technique: For given $(0, 0, 0) \not = (\bs{u}, p_T, p_F) \in \mathcal{X}$, we will find $(\bs{v}, q_T, q_F) \in \mathcal{X}$ such that %\dfn{Rv1-C9, Rv2-SC9(a)}
\begin{align}
\label{eq:inf-sup-sub1} \| (\bs{v}, q_T, q_F) \|_{\mathcal{X}} &\leq C_1 \| (\bs{u}, p_T, p_F) \|_{\mathcal{X}}, \\
\label{eq:inf-sup-sub2} (\tblu{\mathcal{A}} (\bs{u}, p_T, p_F), (\bs{v}, q_T, q_F))_{(\mathcal{X}^*, \mathcal{X})} &\geq C_2 \tblu{\| (\bs{u}, p_T, p_F) \|_{\mathcal{X}}^2} ,
\end{align}
with positive constants $C_1$, $C_2$ which are independent of $\lambda$, $\kappa$, and $\alpha$. 
From these two inequalities we obtain that the desired inf-sup condition holds with $\beta= C_2/C_1$.

Suppose that $(0, 0, 0) \not = (\bs{u}, p_T, p_F) \in \mathcal{X}$ is given. Recall that $p_{T,0}$ is the mean-value zero part of $p_T$. It is well-known from the theory of Stokes equation, cf. \cite[Theorem~5.1]{girault-raviart} that there exists a constant $\beta_0 >0$, depending only on the domain $\Omega$, and $\bs{w} \in \bs{V}$, such that %\dfn{Rv1-C14, Rv2-SC9(b)}
\begin{align} \label{eq:w-estm2}
(\div \bs{w}, p_T) = \| p_{T,0} \|_0^2, \qquad (\e(\bs{w}), \e(\bs{w})) \leq \beta_0^{2} \| p_{T,0} \|_0^2 .
\end{align}
%\dfn{Rv1-C10}
We set $\bs{v} = \tblu{\bs{u} - \delta_0 \bs{w}}$, $q_T = - p_T$, $q_F = -p_F$ with a constant $\delta_0$ which will be determined later. One can check that 
%\dfn{Rv1-C13}
\begin{align} %\label{eq:inf-sup-aux1}
\notag \| (\bs{v}, q_T, q_F) \|_{\mathcal{X}} \leq \tblu{\sqrt{2(1 + \delta_0^2 \beta_0^2)}} \| (\bs{u}, p_T, p_F) \|_{\mathcal{X}} , 
\end{align}
and \eqref{eq:inf-sup-sub1} follows, if  $\delta_0$ is independent of the parameters of our interest. 
%\textcolor{red}{(Continue)}

To establish \eqref{eq:inf-sup-sub2} and determine $\delta_0$, we use the chosen $\bs{v}$, $q_T$, $q_F$, and \eqref{eq:w-estm2} to obtain%\dfn{Rv1-C11, Rv2-SC9(c)}
\begin{align} 
\notag &( {\mathcal{A}} (\bs{u}, p_T, p_F), (\bs{v}, q_T, q_F))_{(\mathcal{X}^*, \mathcal{X})} \\
\notag &= \| \bs{u} \|_{\bs{V}}^2 \tblu{-} \delta_0 (\e(\bs{u}), \e(\bs{w})) + \delta_0 (\div \bs{w}, p_T)  \\
\label{eq:inf-sup-bilinear}& \quad + ( (\lambda^{-1} p_T, p_T) + 2 (\alpha^2 \lambda^{-1} p_F, p_F) - 2 (\alpha \lambda^{-1}p_T, p_F)) + (\kappa \nabla p_F, \nabla p_F)  \\
\notag &= \| \bs{u} \|_{\bs{V}}^2 \tblu{-} \delta_0 (\e(\bs{u}), \e(\bs{w})) + \delta_0 \| p_{T,0} \|_0^2  \\
\notag & \quad + ( ( \lambda^{-1} p_T, p_T) + 2 (\alpha^2 \lambda^{-1} p_F, p_F) - 2 (\alpha \lambda^{-1}p_T, p_F)) + (\kappa \nabla p_F, \nabla p_F) .
\end{align}
By Young's inequality and \eqref{eq:w-estm2}, we also have
\begin{align*}
& (\e(\bs{u}), \e(\bs{w})) \leq \frac{1}{2\theta_0} \| \bs{u} \|_{\bs{V}}^2 + \frac{\theta_0}{2} \| \bs{w} \|_{\bs{V}}^2 \leq \frac{1}{2\theta_0} \| \bs{u} \|_{\bs{V}}^2 + \frac{\theta_0 \beta_0^2}{2} \| p_{T,0} \|_0^2 ,\quad \forall \theta_0 >0.
\end{align*}
Using the above inequality with the choice $\theta_0 = \delta_0 = \beta_0^{-2}$, we derive 
\begin{align} \label{eq:inf-sup-aux2}
\| \bs{u} \|_{\bs{V}}^2 \tblu{-} \delta_0 (\e(\bs{u}), \e(\bs{w})) + \delta_0 \| p_{T,0} \|_0^2 \geq \half \| \bs{u} \|_{\bs{V}}^2 + \frac{\delta_0}{2 } \| p_{T,0} \|_0^2.
\end{align}
Again by Young's inequality, for any $\theta_1 >0$,
\begin{align*}
2 (\alpha \lambda^{-1} p_T, p_F) = 2 (\lambda^{-1/2} p_T, \alpha \lambda^{-1/2} p_F) \leq 2 \theta_1 (\lambda^{-1} p_T, p_T) + \frac{1}{2 \theta_1} (\alpha^2 \lambda^{-1} p_F, p_F) .
\end{align*}
If we take $\theta_1 = 3/8$, then we get 
\begin{multline} \label{eq:inf-sup-aux3}
( ( \lambda^{-1} p_T, p_T) + 2 (\alpha^2 \lambda^{-1} p_F, p_F) - 2 (\alpha \lambda^{-1}p_T, p_F)) \\
\geq  \frac{1}{4} (\lambda^{-1} p_T, p_T) + \frac{2}{3} (\alpha^2 \lambda^{-1} p_F, p_F ) . 
\end{multline}
The inequality \eqref{eq:inf-sup-sub2} is obtained by combining \eqref{eq:inf-sup-bilinear}, \eqref{eq:inf-sup-aux2}, and \eqref{eq:inf-sup-aux3}. Finally, we remark that the constants $C_1$ and $C_2$ in \eqref{eq:inf-sup-sub1}--\eqref{eq:inf-sup-sub2} depend only on $\delta_0$, which is $\beta_0^{-2}$ in the argument, so they are independent of $\lambda$, $\alpha$, $\kappa$. 
\end{proof}
\begin{rmk} \label{rmk:general-bc} \normalfont
The set up in 
Theorem~\ref{thm:3f-inf-sup} is suitable for homogeneous Dirichlet boundary conditions for the displacement $\bs{u}$ and the fluid pressure $p_F$. 
However, similar result can be obtained for more general boundary conditions. For this, we first review possible boundary conditions for Biot's model. Suppose that there are two partitions of $\pd \Omega$,
\begin{align} \label{eq:bdy-decomp}
\pd \Omega = \Gamma_p \cup \Gamma_f, \qquad \pd \Omega = \Gamma_d \cup \Gamma_t,
\end{align}
with $| \Gamma_p |,| \Gamma_d | > 0$, i.e., the Lebesgue measures of $\Gamma_p$ and $\Gamma_d$ are positive. General homogeneous boundary conditions of \eqref{eq:2-field-eq} are given by
\begin{align}
\notag %\label{eq:bc}  
\begin{split}
&p_F(t) = 0 \quad \text{ on } \Gamma_p, \quad - \kappa \nabla p_F(t) \cdot \bs{n} = 0 \quad\text{ on } \Gamma_f, \\
& \bs{u}(t) = 0 \quad\text{ on } \Gamma_d, \quad \ul{\bs{\sigma}}(t) \bs{n} = 0 \quad\text{ on } \Gamma_t,
\end{split}
\end{align}
for time variable $t \in [0, T]$, $T > 0$, in which $\bs{n}$ is the outward unit normal vector field on $\pd \Omega$ and $\ul{\bs{\sigma}}(t) := 2 \mu \e(\bs{u}(t)) + (\lambda \div \bs{u}(t) - \alpha p_F(t)) \ul{\bs{I}}$. The conditions for $p_F$ is a combination of pressure-flux boundary condition as in Darcy flow and the conditions for $\bs{u}$ is a combination of displacement-traction boundary conditions as in elasticity problems. The proper function spaces for this variational formulation are  
\begin{align} \label{eq:space-bc}
\bs{V} := \{ \bs{v} \in \bs{H}^1 \;:\; \bs{v}|_{\Gamma_d} = 0 \}, \quad Q_T = L^2, \quad Q_F := \{ q_F \in {H}^1 \;:\; q_F|_{\Gamma_p} = 0 \}, 
\end{align}
for $\bs{u}$, $p_T$, and $p_F$. 
When $\Gamma_d \not = \pd \Omega$, we choose parameter-dependent norms by 
\begin{align} \label{eq:n3f-norms2}
\| \e(\bs{u}) \|_0 , \quad \| p_{T} \|_0 , \quad \left ( (\alpha^2 \lambda^{-1} p_F, p_F) + (\kappa \nabla p_F, \nabla p_F) \right)^{\half} .
\end{align} 
We can prove a stability result similar to Theorem~\ref{thm:3f-inf-sup} with these norms. In fact, the  proof is easier in this case since the inf-sup condition
\begin{align} \label{eq:stokes-inf-sup}
 \inf_{p_T \in L^2} \sup_{\bs{v} \in \bs{H}_{\Gamma_d}^1} \frac{(\div \bs{v}, p_T)}{\| \bs{v} \|_1 \| p_T \|_0} \geq \beta_0 
\end{align}
holds, and therefore a decomposition of $p_T$, into its mean-value and mean-value zero components, is not necessary. We omit the details since the proof  is completely analogous to the proof of Theorem~\ref{thm:3f-inf-sup}. 
\end{rmk}

\section{Discretization and construction of preconditioners} \label{sec:precond}
In this section we propose finite element discretizations of the three field formulation introduced above,  and show that parameter-robust preconditioners for the discrete problems can be found. In contrast to the discussion above, we will
first consider  problems with general boundary conditions, i.e., boundary conditions with $\Gamma_d \not = \pd \Omega$,
cf. \eqref{eq:bdy-decomp}.
The reason for this reversed order is that the construction of preconditioners in the case when  
$\Gamma_d  = \pd \Omega$, resulting in the choice $\bs{V} = \bs{H}_0^1$, requires a nontrivial technical discussion.

We have shown above that \eqref{eq:n3f-var} is a linear system with parameter-robust stability for the function spaces $\bs{V}$, $Q_T$, $Q_F$ with parameter-dependent norms given by  \eqref{eq:n3f-norms}. If we discretize the system \eqref{eq:n3f-var} with the finite element spaces $\bs{V}_h \subset \bs{V}$, ${Q}_{T,h} \subset Q_T$, $Q_{F,h} \subset Q_F$, then the discrete counterpart of \eqref{eq:n3f-var} is to find $(\bs{u}_h, p_{T,h}, p_{F,h}) \in \bs{V}_h \times {Q}_{T,h} \times Q_{F,h}$ such that 
\begin{align}
\notag (\e(\bs{u}_h), \e(\bs{v})) - (\div \bs{v}, p_{T,h}) &= (\bs{f}, \bs{v}), & & \forall \bs{v} \in \bs{V}_h, \\
\label{eq:n3f-vard} - (\div \bs{u}_h, q_T) - ({\lambda}^{-1} {p}_{T,h}, q_T) + (\alpha \lambda^{-1} p_{F,h}, q_T) &= 0, & & \forall q_T \in Q_{T,h}, \\
\notag (\alpha \lambda^{-1} p_{T,h}, q_F) - 2 (\alpha^2 \lambda^{-1} p_{F,h}, q_F) - (\kappa \nabla p_{F,h}, \nabla q_F) &= (g, q_F), & & \forall q_F \in Q_{F,h}.
\end{align}
A basic stability assumption for this discretization is that the pair  $\bs{V}_h \times Q_{T,h}$    satisfies 
a discrete version of \eqref{eq:stokes-inf-sup}, i.e., 
\begin{align} \label{eq:stokes-inf-sup-h}
 \inf_{p_T \in Q_{T,h}} \sup_{\bs{v} \in \bs{V}_h} \frac{(\div \bs{v}, p_T)}{\| \bs{v} \|_1 \| p_T \|_0} \geq \beta_0 >0,
\end{align}
where $\beta_0$ is independent of $h$. In other words, $\bs{V}_h \times Q_{T,h}$ is a stable Stokes pair.

\begin{theorem} \label{thm:general-bc}
Suppose that $\bs{V}$, $Q_T$, $Q_F$ are as in \eqref{eq:space-bc} with $\Gamma_d \not = \pd \Omega$ and $\Gamma_p$ as in \eqref{eq:bdy-decomp}, and that  $\bs{V}_h \subset \bs{V}$, $Q_{T,h} \subset Q_T$, $Q_{F,h} \subset Q_F$ are corresponding finite element spaces. Furthermore, assume that the pair $\bs{V}_h \times Q_{T,h}$
satisfies the inf-sup condition \eqref {eq:stokes-inf-sup-h}.  Let $\mathcal{X}_h = \bs{V}_h \times Q_{T,h} \times Q_{F,h}$ be the Hilbert space with norm given in \eqref{eq:n3f-norms2}, and ${\mathcal{A}}_h : \mathcal{X}_h \ra \mathcal{X}_h^*$ the operator given by \eqref{eq:n3f-vard}. Then there exists a constant $\beta>0$, such that 
\begin{align} \label{eq:3f-disc-inf-sup}
\inf_{ (\bs{u}, p_T, p_F) \in \mathcal{X}_h } 
\sup_{(\bs{v}, q_T, q_F) \in \mathcal{X}_h } 
\frac{( {\mathcal{A}_h} (\bs{u}, p_T, p_F), (\bs{v}, q_T, q_F))_{(\mathcal{X}_h^*, \mathcal{X}_h)} } 
{\| (\bs{u}, p_T, p_F) \|_{\mathcal{X}_h} \| (\bs{v}, q_T, q_F) \|_{\mathcal{X}_h} }  \geq {\beta} ,
\end{align}
for all parameters $\lambda$, $\alpha$, and $\kappa$ satisfying \eqref{eq:param-range}.
\end{theorem}
We do not prove this result here since the proof   is completely analogous to the proof of Theorem~\ref{thm:3f-inf-sup}. 
We observe that the norms given in \eqref{eq:n3f-norms} shows that a preconditioner of the form
\begin{align} \label{eq:3f-precond2}
\mathcal{B} = 
\begin{pmatrix} 
- \lap^{-1} & 0 & 0 \\
0 & I^{-1}& 0 \\
0 & 0 & \left( \alpha^2\lambda^{-1} I - \div(\kappa \nabla) \right)^{-1}
\end{pmatrix} 
\end{align}
will be a parameter-robust preconditioner.

We now turn to the case with $\Gamma_d = \pd \Omega$ such that $\bs{V} = \bs{H}_0^1$. 
We recall that $L^2_0$ is  the space of $L^2$ functions with mean value zero.
The proper discrete inf-sup condition in this case takes the form 
\begin{align} \label{eq:stokes-inf-sup-h-D}
 \inf_{p_T \in Q_{T,h} \cap L^2_0} \sup_{\bs{v} \in \bs{V}_h} \frac{(\div \bs{v}, p_T)}{\| \bs{v} \|_1 \| p_T \|_0} \geq \beta_0 >0,
\end{align}
where again $\beta_0$ is independent of $h$. 

The following is a discrete analogue of Theorem~\ref{thm:3f-inf-sup}, and its proof is completely analogous to the proof of that theorem.

\begin{theorem} \label{thm:dirichlet-bc}
Suppose that $\bs{V}$, $Q_T$, $Q_F$ are as in \eqref{eq:space-bc} with $\Gamma_d = \pd \Omega$ and $\Gamma_p$ as in \eqref{eq:bdy-decomp}, and that  $\bs{V}_h \subset \bs{V}$, $Q_{T,h} \subset Q_T$, $Q_{F,h} \subset Q_F$ are corresponding finite element spaces. Furthermore, assume that the pair $\bs{V}_h \times Q_{T,h}$ satisfy the inf-sup condition 
\eqref{eq:stokes-inf-sup-h-D}. Let $\mathcal{X}_h = \bs{V}_h \times Q_{T,h} \times Q_{F,h}$ be the Hilbert space with norm 
given in \eqref{eq:n3f-norms}, and ${\mathcal{A}}_h : \mathcal{X}_h \ra \mathcal{X}_h^*$ the operator given by \eqref{eq:n3f-vard}. Then there is a constant $\beta>0$ such that \eqref{eq:3f-disc-inf-sup} holds 
for all parameters  $\lambda$, $\alpha$, and $\kappa$ satisfying \eqref{eq:param-range}.
\end{theorem}

There exist a number of choices of stable Stokes pairs $\bs{V}_h \times Q_{T,h}$, and in Section 6 below we
will present numerical results for two examples, the lowest order Taylor-Hood element and the MINI element. For more examples
of stable Stokes pairs
we refer to \cite{BFBook, girault-raviart}.
The parameter-dependent norms in \eqref{eq:n3f-norms} suggest a block diagonal preconditioner  of the form 
\begin{align} \label{eq:3f-precond}
\mathcal{B} = 
\begin{pmatrix} 
- \lap^{-1} & 0 & 0 \\
0 & \left( {\lambda}^{-1} I + I_0 \right)^{-1}& 0 \\
0 & 0 & \left( {\alpha^2}\lambda^{-1} I - \div(\kappa \nabla) \right)^{-1}
\end{pmatrix} ,
\end{align}
for the continuous system.
We recall that $I$ is the Riesz map of $Q_{T}$ into its dual $Q_{T}^*$, and $I_0$ is the corresponding map into 
the dual of $Q_{T} \cap L_0^2$. The first and third blocks of this block diagonal operator are inverses of standard  second-order elliptic operators, and  corresponding preconditioners to replace the exact inverses in the discrete case are well-studied. In contrast, the operator in the second block is less standard, and, as far as we know, a construction of an effective  preconditioner 
to replace it has not been proposed. We will discuss such a construction below.

\section{A preconditioner for the operator ${\lambda}^{-1} I + I_0 $} \label{sec:weighted-jacobi}
Throughout this section the parameter $\lambda$ is assumed to be a constant.
We recall from the discussion above that in order to construct an effective block diagonal  preconditioner 
of the form \eqref{eq:3f-precond} we need to replace the inverse of the operator ${\lambda}^{-1} I + I_0 $ by a 
spectrally equivalent operator which can be cheaply evaluated. In fact, when $\lambda \ge 1$ the operators 
${\lambda}^{-1} I + I_0 $ and ${\lambda}^{-1} I_m + I_0 $ are spectrally equivalent, so it is enough to 
approximate the inverse of the latter.

Let $N$ be the dimension of $Q_{T,h}$ and $\{ \phi^i \}_{i=1}^N$ be the standard nodal basis of $Q_{T,h}$. Let $1_{\Omega}$ be the constant function on $\Omega$ with value $1/\sqrt{|\Omega|}$ where $|\Omega|$ is the volume of $\Omega$. Denoting the mean-value zero and mean-value parts of $\phi^i$ by $\phi_0^i$ and $\phi_m^i$ as before, we can observe that 
\begin{align} \label{eq:m-i}
\phi_m^i = m_i 1_{\Omega}, \quad m_i := (\phi^i, 1_{\Omega}), \qquad \phi^i = \phi_0^i + \phi_m^i, \qquad \forall 1 \leq i \leq N.
\end{align}
If we let $\Bbb{M}$, $\Bbb{M}_0$, $\Bbb{M}_m$ be mass matrices corresponding to the operators $I$, $I_0$, and $I_m$, then their $(i,j)$-entries are 
\begin{align*}
\Bbb{M}(i,j) := (\phi^i, \phi^j), \quad \Bbb{M}_0 (i,j) := (\phi_0^i, \phi_0^j), \quad \Bbb{M}_m (i,j) := (\phi_m^i, \phi_m^j), \quad \forall 1 \leq i,j \leq N,
\end{align*}
and the matrix corresponding to $\lambda^{-1} I_m + I_0$ is $\lambda^{-1} \Bbb{M}_m + \Bbb{M}_0$.
Since $(\phi^i, \phi^j) = (\phi_0^i + \phi_m^i, \phi_0^j + \phi_m^j) = (\phi_0^i, \phi_0^j) + (\phi_m^i, \phi_m^j)$, one can see $\Bbb{M} = \Bbb{M}_0 + \Bbb{M}_m$, and therefore $\lambda^{-1} \Bbb{M}_m + \Bbb{M}_0 = \Bbb{M} + (\lambda^{-1} - 1) \Bbb{M}_m$. In addition, observe that $\Bbb{M}_m (i,j) = m_i m_j $ by \eqref{eq:m-i}, so
\algn{ \label{eq:m-lam}
\lambda^{-1} \Bbb{M}_m + \Bbb{M}_0 = \Bbb{M} + (\lambda^{-1} - 1) {\bf m} {\bf m}^T,
}
with
\begin{align} \label{eq:m-vec}
{\bf m} = 
\begin{pmatrix}
m_1 \\
m_2 \\
\vdots \\
m_N
\end{pmatrix} .
\end{align}
To construct a preconditioner we need to find an approximate inverse of $\lambda^{-1} \Bbb{M}_m + \Bbb{M}_0$. 
Since $\Bbb{M}$ is positive definite, we can rewrite the right-hand side of \eqref{eq:m-lam} as 
\begin{align} \label{eq:m-lam2}
\Ml :=\Bbb{M} + (\lambda^{-1} - 1) {\bf m} {\bf m}^T = (\Bbb{I} + (\lambda^{-1} - 1) {\bf m} {\bf m}^T \Bbb{M}^{-1}) \Bbb{M}, 
\end{align}
where $\Bbb{I}$ is the $N \times N$ identity matrix. Recall the Sherman--Morrison--Woodbury formula, 
\begin{align} \label{eq:sm-formula}
(\Bbb{I} + {\bf u} {\bf v}^T)^{-1} = \Bbb{I} - \frac{{\bf u v}^T}{1 + {\bf u}^T {\bf v}}, \qquad {\bf u}, {\bf v} \in \R^N \text{ with }{\bf u}^T {\bf v} \not = -1.
\end{align}
We will use it to find the inverse of $\Bbb{I} + (c^{-1} - 1) {\bf m} {\bf m}^T \Bbb{M}^{-1}$ for a constant $c \not = 0$. 
\begin{lemma} \label{lem:mat-id} Let ${\bf w} = (1 \cdots 1)^T \in \R^N$. For the mass matrix $\Bbb{M}$ and $\bf m$ in \eqref{eq:m-vec}, the following two identities holds:
\begin{align} \label{eq:Mx1}
\Bbb{M} 
{\bf w} = 
\sqrt{|\Omega|} {\bf m} , \qquad 
{\bf m}^T 
{\bf w} = \sqrt{|\Omega|}.
\end{align}
\end{lemma}
\begin{proof}
Note that 
\begin{align} \label{eq:phi-sum}
\frac{1}{\sqrt{|\Omega|}} \sum_{j=1}^N \phi^j =  1_{\Omega},
\end{align}
because $\{ \phi^j \}_{1 \leq j \leq N}$ is the standard nodal basis and no boundary condition is imposed on \tblu{$Q_{T,h}$}.% \dfn{Rv2-C14} 

If we consider the $i$-th row of the left-hand side of the first identity in \eqref{eq:Mx1}, then the definition of $\Bbb{M}$, \eqref{eq:phi-sum}, and \eqref{eq:m-i} give
%, one can see that the $i$-th row of the left-hand side of \eqref{eq:Mx1} is 
\begin{align*}
\sum_{j=1}^N \Bbb{M}(i,j) = \sum_{j=1}^N (\phi^i, \phi^j) = \left( \phi^i, \sum_{j=1}^N \phi^j \right) = \sqrt{|\Omega|} m_i .
\end{align*}
This proves the first identity in \eqref{eq:Mx1}. The second identity follows by 
\begin{align*}
\sum_{i=1}^N m_i = \sum_{i=1}^N (\phi^i, 1_{\Omega}) = \sqrt{|\Omega|} (1_{\Omega}, 1_{\Omega}) = \sqrt{|\Omega|}.
\end{align*}
\end{proof}
\begin{cor} 
For the mass matrix $\Bbb{M}$, $\bf m$ in \eqref{eq:m-vec}, $\bf w$ in Lemma~\ref{lem:mat-id} and any constant $c\not =0$, the following holds: %\dfn{Rv2-C15}
\begin{align} \label{eq:m-inv}
(\Bbb{I} + (c^{-1} - 1) {\bf m} {\bf m}^T \Bbb{M}^{-1})^{-1} = \Bbb{I} + (c- 1) \tblu{(\sqrt{|\Omega|})^{-1}} {\bf m} {\bf w}^T .
\end{align}
\end{cor}
\begin{proof}
%\dfn{Rv2-C15}
Since $\Bbb{M}$ is symmetric, the first identity in \eqref{eq:Mx1} gives ${\bf m}^T \Bbb{M}^{-1} = (\sqrt{|\Omega|})^{-1} \tblu{{\bf w}^T}$. If we set ${\bf u} = (c^{-1} - 1) {\bf m}$ and $\tblu{{\bf v}^T} = {\bf m}^T \Bbb{M}^{-1}$, then the second identity in \eqref{eq:Mx1} gives ${\bf u}^T{\bf v} = c^{-1} - 1 \not = -1$. The assertion follows from \eqref{eq:sm-formula}.
\end{proof}

\begin{theorem}
Let $\Vl = (\Bbb{I}+a {\bf m} {\bf w}^T)^{-1}$ with $a = (-1 + \sqrt{\lambda})/\sqrt{|\Omega|}$. Then, for $\Ml$ in \eqref{eq:m-lam2}, $\Ml = \Vl \Bbb{M} \Vl^T$.
Thus, if $\Bbb{D}$ is a preconditioner of $\Bbb{M}$ with condition number $K$, then $\Vl^{-T} \Bbb{D} \Vl^{-1}$ is a preconditioner of $\Ml$ with same condition number.
\end{theorem}
\begin{proof}
From the definition of $\Vl$ and the second identity in \eqref{eq:Mx1}, we can see 
\algns{ \Vl^{-2} = (\Bbb{I} + a {\bf m} {\bf w}^T)^2 = \Bbb{I} + (2a + a^2 \sqrt{|\Omega|}) {\bf m}{\bf w}^T = \Bbb{I} + (\lambda-1)(\sqrt{|\Omega|})^{-1} {\bf m}{\bf w}^T .
}
By the identity \eqref{eq:m-inv} with $c = \lambda$ and \eqref{eq:m-lam2}, we have $\Ml = \Vl^2 \Bbb{M}$. If we use \eqref{eq:m-inv} for $\Vl^{-1}$, then one can verify that $\Vl = \Bbb{I} + \bar{a} {\bf m}{\bf m}^T \Bbb{M}^{-1}$ with $\bar{a} = 1-(\sqrt{\lambda})^{-1}$. From this expression of $\Vl$, it is easy to check that $\Vl \Bbb{M} = \Bbb{M} \Vl^{T}$, so $\Ml = \Vl^2 \Bbb{M} = \Vl \Bbb{M} \Vl^T$. The assertion for preconditioner $\Vl^{-T} \Bbb{D} \Vl^{-1}$ follows from the identity $\Ml \Vl^{-T} \Bbb{D} \Vl^{-1} \Ml = \Vl \Bbb{M D M} \Vl^T$.
\end{proof}
For the preconditioner $\Bbb{D}$ for $\Bbb{M}$, it is known that the Jacobi preconditioner, i.e., the inverse of diagonal of mass matrix as a preconditioner has explicit condition number bounds \cite{MR968517}. 
If $Q_{T,h}$ is the piecewise linear continuous finite element, then the Jacobi preconditioner $\Bbb{D}$ for $\Bbb{M}$ is a constant multiple of the diagonal matrix $\diag (m_1^{-1}, m_2^{-1}, ..., m_N^{-1})$. As a consequence, ${\bf w}{\bf m}^T \Bbb{D} = \Bbb{D} {\bf m}{\bf w}^T$, so $\Vl^{-T} \Bbb{D} \Vl^{-1}$ can be reduced to $\Bbb{D} \Vl^{-2} = \Bbb{D} (\Bbb{I} + (\lambda - 1)(\sqrt{|\Omega|})^{-1} {\bf m} {\bf w}^T)$.
% The identity \eqref{eq:m-inv} gives a way to construct a $\lambda$-robust preconditioner of $\lambda^{-1} \Bbb{M}_m + \Bbb{M}_0$. If $\Bbb{D}$ is a good preconditioner of $\Bbb{M}$, then $\Bbb{D} (\Bbb{I} + (\lambda - 1) {\bf m} {\bf w}^T)$ will be a good preconditioner of $\lambda^{-1} \Bbb{M}_m + \Bbb{M}_0$ because 
% \begin{align*}
% &\Bbb{D} (\Bbb{I} + (\lambda - 1) {\bf m} {\bf w}^T) (\lambda^{-1} \Bbb{M}_m + \Bbb{M}_0) \\
% &= \Bbb{D} (\Bbb{I} + (\lambda - 1) {\bf m} {\bf w}^T) (\Bbb{M} + (\lambda^{-1} - 1) {\bf m} {\bf m}^T ) \\
% &= \Bbb{D} (\Bbb{I} + (\lambda - 1) {\bf m} {\bf w}^T) (\Bbb{I} + (\lambda^{-1} - 1) {\bf m} {\bf m}^T \Bbb{M}^{-1}) \Bbb{M} \\
% &= \Bbb{D} \Bbb{M} .
% \end{align*}

There is one caution in the implementation of the preconditioner $\Vl^{-T} \Bbb{D} \Vl^{-1}$ because ${\bf m w}^T$ and ${\bf m} {\bf m}^T$ in $\Vl^{-1}$ and $\Ml$ are dense matrices in general. We remark that the minimum residual method requires only matrix-vector multiplication operations. Therefore, to avoid computation with these dense matrices, %${\bf m w}^T$ and ${\bf m} {\bf m}^T$, 
we use the structure of the matrix ${\bf m w}^T$. More precisely, ${\bf m w}^T {\bf v}$ for an $\R^N$-vector $\bf v$ can be computed with two operations, the inner product ${\bf w}^T {\bf v}$ and constant-vector multiplication $({\bf w}^T {\bf v}){\bf m}$. Similarly, we can avoid generating ${\bf m} {\bf m}^T$ in $\Ml= \Bbb{M} + (\lambda^{-1} - 1) {\bf m} {\bf m}^T$. 
Finally, we remark that the preconditioner $\Vl^{-T} \Bbb{D} \Vl^{-1}$ is useful when a piecewise discontinuous finite element is used for $Q_{T,h}$ because ${\bf m w}^T$ and $\Ml$ are not sparse.
\begin{table}[ht]  
\caption{Boundary conditions (BC), preconditioners (PC), and finite elements of test cases. The first three cases use the lowest order Taylor--Hood element and the last case uses the MINI element ($\bs{B}$ = vector-valued bubble function).} \label{table:cases}
\centering
\begin{tabular}{>{\small}c | >{\small}c >{\small}c >{\small}c >{\small}c >{\small}c >{\small}c}
\hline
& \multirow{2}{*}{BC} & \multirow{2}{*}{PC} &  \multirow{2}{*}{finite elements}& \multicolumn{3}{>{\small}c}{system size for $N$} \\
& & & & $32$ & $64$ & $128$ \\  \hline %\hline
Case 1 & $\Gamma_d \not = \pd \Omega$, $\Gamma_p = \pd \Omega$ & \eqref{eq:3f-precond2}	&  $\bs{\mathcal{P}}_2$-$\mathcal{P}_1$-$\mathcal{P}_1$ & 10628 & 41732 & 165380 \\ 
Case 2 & $\Gamma_d = \pd \Omega$, $\Gamma_p = \pd \Omega$ & \eqref{eq:3f-precond}	&  $\bs{\mathcal{P}}_2$-$\mathcal{P}_1$-$\mathcal{P}_1$ & 10628 & 41732 & 165380 \\ 
Case 3 & $\Gamma_d = \pd \Omega$, $\Gamma_p = \pd \Omega$ & \eqref{eq:3f-precond2}	& $\bs{\mathcal{P}}_2$-$\mathcal{P}_1$-$\mathcal{P}_1$ & 10628 & 41732 & 165380 \\ 
Case 4 & $\Gamma_d \not = \pd \Omega$, $\Gamma_p = \pd \Omega$ & \eqref{eq:3f-precond2}	& $(\bs{\mathcal{P}}_1+\bs{B})$-$\mathcal{P}_1$-$\mathcal{P}_1$ & 8452& 33284 & 132100 \\ 
\hline
\end{tabular}                           
\end{table}

\begin{table}[th]  %file 'biot_new3f_AMGTH1'
\caption{\small{Numbers of iteration and condition numbers of Case 1 (cf. Table~\ref{table:cases}). ($\Omega$ = unit square, partitioned as bisections of $N \times N$ rectangles, convergence criterion with relative residual of $10^{-6}$})} \label{table:eg1a}
\centering
\begin{tabular}{>{\small}c | >{\small}c | >{\small}c || >{\small}c >{\small}c >{\small}c >{\small}c }
\hline
& & & \multicolumn{4}{>{\small}c}{$\kappa$} \\
$N$ & $\alpha$ & $\lambda$ & $10^0$	&$10^{-4}$  &  $10^{-8}$ & $10^{-12}$ \\  
\hline \hline
\multirow{9}{*} {$32$} & 
\multirow{3}{*}{$10^0$} 
& $10^0$ & $33 \;(3.8) $ & $43 \;(6.3) $& $47 \;(7.6) $ & $47 \;(7.6) $  \\ 
% & $10^2$ & $51 \;(20.2)$ & $58 \;(20.2)$& $65 \;(20.2)$ & $61 \;(20.2)$  \\ 
& & $10^4$ & $52 \;(21.7)$ & $52 \;(21.7)$& $65 \;(21.7)$ & $63 \;(21.7)$  \\ 
% & $10^6$ & $52 \;(21.7)$ & $50 \;(21.7)$& $59 \;(21.7)$ & $63 \;(21.7)$  \\ 
& & $10^8$ & $52 \;(21.7)$ & $54 \;(21.7)$& $52 \;(21.7)$ & $62 \;(21.7)$  \\ 
\cline{3-7}
& \multirow{3}{*}{$10^{-2}$} 
& $10^0$ & $33 \;(3.8) $ & $33 \;(3.9) $& $43 \;(6.3) $  & $47 \;(7.6) $  \\ 
% & $10^2$ & $52 \;(20.2)$ & $52 \;(20.2)$& $58 \;(20.2)$  & $56 \;(20.2)$  \\ 
& & $10^4$ & $52 \;(21.7)$ & $52 \;(21.7)$& $52 \;(21.7)$  & $63 \;(21.7)$  \\ 
% & $10^6$ & $54 \;(21.7)$ & $52 \;(21.7)$& $52 \;(21.7)$  & $57 \;(21.7)$  \\ 
& & $10^8$ & $52 \;(21.7)$ & $52 \;(21.7)$& $52 \;(21.7)$  & $52 \;(21.7)$  \\ 
\cline{3-7}
& \multirow{3}{*}{$10^{-4}$} 
& $10^0$ & $33 \;(3.8) $ & $33 \;(3.8) $& $33 \;(3.8) $  & $43 \;(6.3) $  \\ 
% & $10^2$ & $51 \;(20.2)$ & $52 \;(20.2)$& $52 \;(20.2)$  & $58 \;(20.2)$  \\ 
& & $10^4$ & $50 \;(21.7)$ & $52 \;(21.7)$& $50 \;(21.7)$  & $52 \;(21.7)$  \\ 
% & $10^6$ & $52 \;(21.7)$ & $50 \;(21.7)$& $52 \;(21.7)$  & $54 \;(21.7)$  \\ 
& & $10^8$ & $54 \;(21.7)$ & $52 \;(21.7)$& $52 \;(21.7)$  & $52 \;(21.7)$  \\ 
\cline{2-7}                                                 
\multirow{9}{*} {$64$} & 
\multirow{3}{*}{$10^{0}$} 
& $10^0$ & $33 \;(3.9) $ & $40 \;(5.6) $& $47 \;(7.6) $  & $47 \;(7.6) $  \\ 
% & $10^2$ & $46 \;(20.2)$ & $57 \;(20.2)$& $58 \;(20.2)$  & $63 \;(20.2)$  \\ 
& & $10^4$ & $52 \;(21.7)$ & $52 \;(21.7)$& $63 \;(21.7)$  & $63 \;(21.7)$  \\ 
% & $10^6$ & $52 \;(21.7)$ & $52 \;(21.7)$& $57 \;(21.7)$  & $62 \;(21.7)$  \\ 
& & $10^8$ & $46 \;(21.7)$ & $52 \;(21.7)$& $52 \;(21.7)$  & $62 \;(21.7)$  \\ 
\cline{3-7}
& \multirow{3}{*}{$10^{-2}$} 
& $10^0$ & $33 \;(3.8) $ & $33 \;(3.8) $& $40 \;(5.6) $  & $47 \;(7.6) $  \\ 
% & $10^2$ & $51 \;(20.2)$ & $52 \;(20.2)$& $55 \;(20.2)$  & $63 \;(20.2)$  \\ 
& & $10^4$ & $52 \;(21.7)$ & $52 \;(21.7)$& $52 \;(21.7)$  & $58 \;(21.7)$  \\ 
% & $10^6$ & $48 \;(21.7)$ & $52 \;(21.7)$& $52 \;(21.7)$  & $57 \;(21.7)$  \\ 
& & $10^8$ & $52 \;(21.7)$ & $46 \;(21.7)$& $52 \;(21.7)$  & $48 \;(21.7)$  \\ 
\cline{3-7}
& \multirow{3}{*}{$10^{-4}$} 
& $10^0$ & $33 \;(3.9) $ & $33 \;(3.8) $& $33 \;(3.8) $  & $40 \;(5.6) $  \\ 
% & $10^2$ & $52 \;(20.2)$ & $46 \;(20.2)$& $52 \;(20.2)$  & $55 \;(20.2)$  \\ 
& & $10^4$ & $50 \;(21.7)$ & $46 \;(21.7)$& $50 \;(21.7)$  & $52 \;(21.7)$  \\ 
% & $10^6$ & $52 \;(21.7)$ & $52 \;(21.7)$& $52 \;(21.7)$  & $52 \;(21.7)$  \\ 
& & $10^8$ & $52 \;(21.7)$ & $52 \;(21.7)$& $50 \;(21.7)$  & $52 \;(21.7)$  \\ 
\cline{2-7}
\multirow{9}{*} {$128$} & 
\multirow{3}{*}{$10^{0}$} 
& $10^0$ & $32 \;(3.8) $ & $39 \;(5.4) $& $46 \;(7.7) $ & $45 \;(7.7) $  \\ 
% & $10^2$ & $48 \;(20.2)$ & $55 \;(20.2)$& $56 \;(20.2)$ & $61 \;(20.2)$  \\ 
& & $10^4$ & $51 \;(21.7)$ & $52 \;(21.7)$& $61 \;(21.7)$ & $59 \;(21.7)$  \\ 
% & $10^6$ & $50 \;(21.7)$ & $50 \;(21.7)$& $55 \;(21.7)$ & $59 \;(21.7)$  \\ 
& & $10^8$ & $51 \;(21.7)$ & $50 \;(21.7)$& $50 \;(21.7)$ & $54 \;(21.7)$  \\ 
\cline{3-7}
& \multirow{3}{*}{$10^{-2}$} 
& $10^0$ & $33 \;(3.8) $ & $33 \;(3.8) $& $39 \;(5.2) $ & $46 \;(7.7) $  \\ 
% & $10^2$ & $50 \;(20.2)$ & $47 \;(20.2)$& $55 \;(20.2)$ & $58 \;(20.2)$  \\ 
& & $10^4$ & $48 \;(21.7)$ & $45 \;(21.7)$& $52 \;(21.7)$ & $58 \;(21.7)$  \\ 
% & $10^6$ & $50 \;(21.7)$ & $50 \;(21.7)$& $46 \;(21.7)$ & $52 \;(21.7)$  \\ 
& & $10^8$ & $50 \;(21.7)$ & $52 \;(21.7)$& $50 \;(21.7)$ & $52 \;(21.7)$  \\ 
\cline{3-7}
& \multirow{3}{*}{$10^{-4}$} 
& $10^0$ & $33 \;(3.8) $ & $32 \;(3.8) $& $33 \;(3.8) $ & $39 \;(5.2) $  \\ 
% & $10^2$ & $50 \;(20.2)$ & $50 \;(20.2)$& $48 \;(20.2)$ & $49 \;(20.2)$  \\ 
& & $10^4$ & $44 \;(21.7)$ & $52 \;(21.7)$& $52 \;(21.7)$ & $52 \;(21.7)$  \\ 
% & $10^6$ & $46 \;(21.7)$ & $50 \;(21.7)$& $52 \;(21.7)$ & $46 \;(21.7)$  \\ 
& & $10^8$ & $50 \;(21.7)$ & $48 \;(21.7)$& $48 \;(21.7)$ & $50 \;(21.7)$  \\ 
\hline
\end{tabular}            
\end{table}

\section{Numerical results}
In this section we present some  numerical results which illustrate our theoretical results for the proposed preconditioners \eqref{eq:3f-precond2}
and \eqref{eq:3f-precond}. As before, all numerical experiments are carried out using FEniCS with Hypre algebraic multigrid operators 
as replacements for the exact inverses appearing in the first and third block of \eqref{eq:3f-precond2}
and \eqref{eq:3f-precond}. Furthermore, in  the preconditioner of the form \eqref{eq:3f-precond}, the second block is constructed by using the technique 
 outlined in Section~\ref{sec:weighted-jacobi}, while the standard Jacobi preconditioner is used in the second block of \eqref{eq:3f-precond2}.

In all the experiments the domain $\Omega$ is unit square. We show numerical results for four different combinations of boundary conditions, finite element spaces, and preconditioners. The different combinations are presented as 
Case 1-4 in  Table~\ref{table:cases}. In Case 1 and 4 the statement $\Gamma_d \not = \pd \Omega$ means that 
$\Gamma_d$ is taken as in Example 3.1, while problems with $\Gamma_d = \pd \Omega$ are consider in 
Case 2 and Case 3. We compare numerical results obtained by the  two preconditioners with structure of the form \eqref{eq:3f-precond} and \eqref{eq:3f-precond2}. The result of Theorem~\ref{thm:dirichlet-bc} suggests that preconditioners of the form \eqref{eq:3f-precond} are more robust than the ones of the form \eqref{eq:3f-precond2} in the case of Dirichlet boundary conditions. In other words, we expect that the results of Case 2 are more robust than the ones of Case 3. However, the  system preconditioned with a preconditioner of the form \eqref{eq:3f-precond2} has only one bad eigenvalue.
Therefore, as in Example~\ref{thm:eg2}, we can expect small  differences in the number of iterations.   Finally, in Case 4 we use the MINI element instead of the Taylor--Hood element  in order to show that 
our results are robust with respect to the choice of finite element spaces, as long as they fulfill the assumptions of the theory. 
In most of the examples the parameters $\lambda$, $\alpha$ and $\kappa$ are taken to be constants. However, in the last experiment, presented in Table~\ref{table:varying-kap}, $\kappa$ varies with the spatial variable.

\begin{table}[t]  % file 'biot_new3fTH1', 'biot_new3fTH2', 'biot_new3fTH3', 'biot_new3fMINI'
\caption{Numbers of iteration of test cases in Table~\ref{table:cases}. ($\Omega$ = unit square, partitioned into bisections of $N \times N$ rectangles, convergence criterion with relative residual of $10^{-6}$)} \label{table:4case-result}
\centering
\begin{tabular}{>{\small}c | >{\small}c | >{\small}c || >{\small}c >{\small}c >{\small}c | >{\small}c >{\small}c >{\small}c | >{\small}c >{\small}c >{\small}c | >{\small}c >{\small}c >{\small}c }
\hline
& & & \multicolumn{3}{>{\small}c}{Case 1} & \multicolumn{3}{>{\small}c}{Case 2} & \multicolumn{3}{>{\small}c}{Case 3} & \multicolumn{3}{>{\small}c}{Case 4} \\ \hline %& \multirow{2}{*}{order } & \multirow{2}{*}{$n$ } \\ 
& & & \multicolumn{3}{>{\small}c|}{$N$} & \multicolumn{3}{>{\small}c|}{$N$} & \multicolumn{3}{>{\small}c|}{$N$} & \multicolumn{3}{>{\small}c}{$N$} \\
$\kappa$& $\alpha$ & $\lambda$ &   $32$	&$64$ &$128$ & $32$ &$64$ &$128$ & $32$ &$64$ &$128$ & $32$ &$64$ &$128$\\  \hline \hline
\multirow{6}{*}{$10^0$} &\multirow{3}{*}{$10^0$} & $10^0$ 	& $33$	& $33$	& $32$ & $29$	& $29$	& $29$ & $29$	& $29$	& $29$ & $34$	& $34$	& $34$\\ 
& & $10^{4}$ 							& $52$	& $52$	& $51$ & $46$	& $46$	& $46$ & $66$	& $45$	& $44$ & $60$	& $61$	& $60$\\ 
& & $10^{8}$ 							& $52$	& $46$	& $51$ & $46$	& $46$	& $45$ & $44$	& $44$	& $43$ & $60$	& $61$	& $60$\\ 
\cline{3-15}
& \multirow{3}{*}{$10^{-4}$} & $10^0$				& $33$	& $33$	& $33$ & $29$	& $29$	& $29$ & $29$	& $29$	& $29$ & $34$	& $34$	& $34$\\ 
& & $10^{4}$ 							& $50$	& $50$	& $44$ & $46$	& $46$	& $44$ & $66$	& $44$	& $44$ & $60$	& $60$	& $60$\\ 
& & $10^{8}$ 							& $54$	& $52$	& $50$ & $46$	& $46$	& $45$ & $45$	& $44$	& $43$ & $60$	& $60$	& $60$\\ 
\cline{2-15}
\multirow{6}{*}{$10^{-4}$} &\multirow{3}{*}{$10^0$} & $10^0$ 	& $43$	& $40$	& $39$ & $39$	& $38$	& $36$ & $39$	& $38$	& $36$ & $46$	& $43$	& $40$\\ 
& & $10^{4}$ 							& $52$	& $52$	& $52$ & $46$	& $46$	& $45$ & $68$	& $44$	& $44$ & $60$	& $61$	& $60$\\ 
& & $10^{8}$ 							& $54$	& $52$	& $50$ & $46$	& $46$	& $45$ & $66$	& $44$	& $43$ & $60$	& $60$	& $60$\\ 
\cline{3-15}
& \multirow{3}{*}{$10^{-4}$} & $10^0$				& $33$	& $33$	& $32$ & $29$	& $29$	& $29$ & $29$	& $29$	& $29$ & $34$	& $34$	& $34$\\ 
& & $10^{4}$ 							& $52$	& $46$	& $52$ & $46$	& $46$	& $45$ & $66$	& $44$	& $42$ & $60$	& $60$	& $60$\\ 
& & $10^{8}$ 							& $52$	& $52$	& $48$ & $46$	& $46$	& $45$ & $44$	& $44$	& $43$ & $60$	& $61$	& $60$\\ 
\cline{2-15}
\multirow{6}{*}{$10^{-8}$} &\multirow{3}{*}{$10^0$} & $10^0$ 	& $47$	& $47$	& $46$ & $42$	& $42$	& $42$ & $42$	& $42$	& $42$ & $52$	& $52$	& $52$\\ 
& & $10^{4}$ 							& $65$	& $63$	& $61$ & $61$	& $59$	& $58$ & $60$	& $58$	& $57$ & $73$	& $73$	& $72$\\ 
& & $10^{8}$ 							& $52$	& $52$	& $50$ & $46$	& $46$	& $45$ & $44$	& $44$	& $43$ & $60$	& $60$	& $60$\\ 
\cline{3-15}
& \multirow{3}{*}{$10^{-4}$} & $10^0$				& $47$	& $33$	& $33$ & $29$	& $29$	& $29$ & $29$	& $29$	& $29$ & $34$	& $34$	& $34$\\ 
& & $10^{4}$ 							& $65$	& $50$	& $52$ & $46$	& $45$	& $45$ & $66$	& $44$	& $43$ & $60$	& $60$	& $60$\\ 
& & $10^{8}$ 							& $52$	& $50$	& $48$ & $46$	& $46$	& $44$ & $44$	& $44$	& $40$ & $60$	& $60$	& $60$\\ 
\cline{2-15}
\multirow{6}{*}{$10^{-12}$} &\multirow{3}{*}{$10^0$} & $10^0$ 	& $47$	& $47$	& $45$ & $42$	& $42$	& $42$ & $42$	& $42$	& $42$ & $52$	& $52$	& $52$\\ 
& & $10^{4}$ 							& $63$	& $63$	& $59$ & $58$	& $58$	& $57$ & $57$	& $56$	& $56$ & $72$	& $72$	& $72$\\ 
& & $10^{8}$ 							& $62$	& $62$	& $54$ & $58$	& $58$	& $50$ & $57$	& $56$	& $54$ & $72$	& $72$	& $71$\\ 
\cline{3-15}
& \multirow{3}{*}{$10^{-4}$} & $10^0$				& $43$	& $40$	& $39$ & $39$	& $38$	& $36$ & $37$	& $38$	& $36$ & $47$	& $43$	& $40$\\ 
& & $10^{4}$ 							& $52$	& $52$	& $52$ & $46$	& $46$	& $44$ & $66$	& $44$	& $44$ & $60$	& $61$	& $60$\\ 
& & $10^{8}$ 							& $52$	& $52$	& $50$ & $46$	& $45$	& $44$ & $44$	& $43$	& $43$ & $60$	& $61$	& $60$\\ 
\hline
\end{tabular}                           
\end{table}

\begin{table}[h!]  %file 'biot_new3f_AMGkap'
\caption{\small{Numbers of iteration and condition numbers of Case 1 (cf. Table~\ref{table:cases}) with nonconstant $\kappa$. ($\Omega$ = unit square, partitioned as bisections of $N \times N$ rectangles, $\Omega_1 = \{ (x,y)\,:\, 0 \leq x \leq 1, 1/4 \leq y \leq 3/4\}$, $\kappa = 1$ on $\Omega \setminus \Omega_1$, convergence criterion with relative residual of $10^{-6}$})} \label{table:varying-kap}
\centering
\begin{tabular}{>{\small}c | >{\small}c | >{\small}c || >{\small}c >{\small}c >{\small}c >{\small}c >{\small}c  }
\hline
& & & \multicolumn{5}{>{\small}c}{$\kappa$ on $\Omega_1$ } \\
$N$ & $\alpha$ & $\lambda$ & $10^{-2}$	&$10^{-4}$  &  $10^{-6}$ & $10^{-8}$ &  $10^{-10}$  \\  
\hline \hline
\multirow{9}{*} {$32$} & 
\multirow{3}{*}{$10^0$} 
& $10^0$ & $34 \;(4.5) $ & $42 \;(6.1) $& $46 \;(7.4) $ & $46 \;(7.5) $  & $46 \;(7.4) $  \\ 
% & $10^2$ & $52 \;(20.2)$ & $56 \;(20.2)$& $65 \;(20.2)$ & $59 \;(20.2)$  & $60 \;(20.2)$  \\ 
& & $10^4$ & $54 \;(21.7)$ & $52 \;(21.7)$& $56 \;(21.7)$ & $65 \;(21.7)$  & $61 \;(21.7)$  \\ 
% & $10^6$ & $52 \;(21.7)$ & $54 \;(21.7)$& $52 \;(21.7)$ & $55 \;(21.7)$  & $63 \;(21.7)$  \\ 
& & $10^8$ & $54 \;(21.7)$ & $54 \;(21.7)$& $46 \;(21.7)$ & $52 \;(21.7)$  & $55 \;(21.7)$  \\ 
\cline{3-8}
& \multirow{3}{*}{$10^{-2}$} 
& $10^0$ & $33 \;(3.8) $ & $33 \;(3.8) $& $34 \;(4.4) $  & $42 \;(6.1) $ & $46 \;(7.4) $   \\ 
% & $10^2$ & $52 \;(20.2)$ & $52 \;(20.2)$& $52 \;(20.2)$  & $54 \;(20.2)$ & $64 \;(20.2)$   \\ 
& & $10^4$ & $52 \;(21.7)$ & $48 \;(21.7)$& $50 \;(21.7)$  & $53 \;(21.7)$ & $56 \;(21.7)$   \\ 
% & $10^6$ & $50 \;(21.7)$ & $52 \;(21.7)$& $48 \;(21.7)$  & $52 \;(21.7)$ & $52 \;(21.7)$   \\ 
& & $10^8$ & $52 \;(21.7)$ & $52 \;(21.7)$& $50 \;(21.7)$  & $52 \;(21.7)$ & $46 \;(21.7)$   \\ 
\cline{3-8}
& \multirow{3}{*}{$10^{-4}$} 
& $10^0$ & $33 \;(3.8) $ & $33 \;(3.8) $& $33 \;(3.8) $  & $33 \;(3.8) $ & $34 \;(4.5) $   \\ 
% & $10^2$ & $52 \;(20.2)$ & $52 \;(20.2)$& $52 \;(20.2)$  & $52 \;(20.2)$ & $52 \;(20.2)$   \\ 
& & $10^4$ & $52 \;(21.7)$ & $52 \;(21.7)$& $52 \;(21.7)$  & $52 \;(21.7)$ & $52 \;(21.7)$   \\ 
% & $10^6$ & $54 \;(21.7)$ & $50 \;(21.7)$& $52 \;(21.7)$  & $52 \;(21.7)$ & $52 \;(21.7)$   \\ 
& & $10^8$ & $52 \;(21.7)$ & $52 \;(21.7)$& $50 \;(21.7)$  & $53 \;(21.7)$ & $50 \;(21.7)$   \\ 
\cline{2-8}                                                                              
\multirow{9}{*} {$64$} & 
\multirow{3}{*}{$10^{0}$} 
& $10^0$ & $34 \;(4.5) $ & $39 \;(5.4) $& $44 \;(7.4) $  & $46 \;(7.5) $ & $46 \;(7.5) $   \\ 
% & $10^2$ & $50 \;(20.2)$ & $54 \;(20.2)$& $63 \;(20.2)$  & $60 \;(20.2)$ & $59 \;(20.2)$   \\ 
& & $10^4$ & $52 \;(21.7)$ & $52 \;(21.7)$& $55 \;(21.7)$  & $63 \;(21.7)$ & $60 \;(21.7)$   \\ 
% & $10^6$ & $52 \;(21.7)$ & $50 \;(21.7)$& $50 \;(21.7)$  & $54 \;(21.7)$ & $63 \;(21.7)$   \\ 
& & $10^8$ & $52 \;(21.7)$ & $46 \;(21.7)$& $46 \;(21.7)$  & $50 \;(21.7)$ & $51 \;(21.7)$   \\ 
\cline{3-8}
& \multirow{3}{*}{$10^{-2}$} 
& $10^0$ & $33 \;(3.8) $ & $33 \;(3.8) $& $34 \;(4.4) $  & $39 \;(5.4) $ & $45 \;(7.4) $   \\ 
% & $10^2$ & $52 \;(20.2)$ & $52 \;(20.2)$& $51 \;(20.2)$  & $53 \;(20.2)$ & $63 \;(20.2)$   \\ 
& & $10^4$ & $46 \;(21.7)$ & $52 \;(21.7)$& $52 \;(21.7)$  & $52 \;(21.7)$ & $51 \;(21.7)$   \\ 
% & $10^6$ & $46 \;(21.7)$ & $52 \;(21.7)$& $52 \;(21.7)$  & $48 \;(21.7)$ & $52 \;(21.7)$   \\ 
& & $10^8$ & $48 \;(21.7)$ & $52 \;(21.7)$& $52 \;(21.7)$  & $52 \;(21.7)$ & $52 \;(21.7)$   \\ 
\cline{3-8}
& \multirow{3}{*}{$10^{-4}$} 
& $10^0$ & $33 \;(3.8) $ & $33 \;(3.8) $& $33 \;(3.8) $  & $33 \;(3.8) $ & $34 \;(4.4) $   \\ 
% & $10^2$ & $46 \;(20.2)$ & $52 \;(20.2)$& $48 \;(20.2)$  & $52 \;(20.2)$ & $52 \;(20.2)$   \\ 
& & $10^4$ & $52 \;(21.7)$ & $52 \;(21.7)$& $52 \;(21.7)$  & $52 \;(21.7)$ & $52 \;(21.7)$   \\ 
% & $10^6$ & $52 \;(21.7)$ & $50 \;(21.7)$& $52 \;(21.7)$  & $52 \;(21.7)$ & $52 \;(21.7)$   \\ 
& & $10^8$ & $52 \;(21.7)$ & $52 \;(21.7)$& $52 \;(21.7)$  & $52 \;(21.7)$ & $46 \;(21.7)$   \\ 
\cline{2-8}
\multirow{9}{*} {$128$} & 
\multirow{3}{*}{$10^{0}$} 
& $10^0$ & $34 \;(4.4) $ & $37 \;(5.0) $& $44 \;(7.2) $ & $46 \;(7.5) $ & $44 \;(7.5) $  \\ 
% & $10^2$ & $50 \;(20.2)$ & $53 \;(20.2)$& $61 \;(20.2)$ & $58 \;(20.2)$ & $59 \;(20.2)$  \\ 
& & $10^4$ & $50 \;(21.7)$ & $48 \;(21.7)$& $49 \;(21.7)$ & $56 \;(21.7)$ & $61 \;(21.7)$  \\ 
% & $10^6$ & $52 \;(21.7)$ & $51 \;(21.7)$& $51 \;(21.7)$ & $49 \;(21.7)$ & $60 \;(21.7)$  \\ 
& & $10^8$ & $50 \;(21.7)$ & $50 \;(21.7)$& $52 \;(21.7)$ & $50 \;(21.7)$ & $47 \;(21.7)$  \\ 
\cline{3-8}
& \multirow{3}{*}{$10^{-2}$} 
& $10^0$ & $33 \;(3.8) $ & $32 \;(3.9) $& $34 \;(4.4) $ & $37 \;(5.0) $ & $44 \;(7.1) $  \\ 
% & $10^2$ & $51 \;(20.2)$ & $49 \;(20.2)$& $50 \;(20.2)$ & $48 \;(20.2)$ & $58 \;(20.2)$  \\ 
& & $10^4$ & $50 \;(21.7)$ & $46 \;(21.7)$& $50 \;(21.7)$ & $52 \;(21.7)$ & $49 \;(21.7)$  \\ 
% & $10^6$ & $52 \;(21.7)$ & $52 \;(21.7)$& $46 \;(21.7)$ & $48 \;(21.7)$ & $46 \;(21.7)$  \\ 
& & $10^8$ & $52 \;(21.7)$ & $52 \;(21.7)$& $50 \;(21.7)$ & $52 \;(21.7)$ & $50 \;(21.7)$  \\ 
\cline{3-8}
& \multirow{3}{*}{$10^{-4}$} 
& $10^0$ & $32 \;(3.8) $ & $33 \;(3.9) $& $33 \;(3.8) $ & $33 \;(3.9) $ & $34 \;(4.4) $  \\ 
% & $10^2$ & $49 \;(20.2)$ & $51 \;(20.2)$& $48 \;(20.2)$ & $51 \;(20.2)$ & $50 \;(20.2)$  \\ 
& & $10^4$ & $52 \;(21.7)$ & $50 \;(21.7)$& $52 \;(21.7)$ & $49 \;(21.7)$ & $46 \;(21.7)$  \\ 
% & $10^6$ & $52 \;(21.7)$ & $52 \;(21.7)$& $50 \;(21.7)$ & $46 \;(21.7)$ & $50 \;(21.7)$  \\ 
& & $10^8$ & $52 \;(21.7)$ & $50 \;(21.7)$& $48 \;(21.7)$ & $52 \;(21.7)$ & $50 \;(21.7)$  \\ 
\hline
\end{tabular}                           
\end{table}

In Table~\ref{table:eg1a}, we present numbers of iteration of Case 1. The results are fairly robust with respect to  parameter changes and mesh refinements. To compare robustness of preconditioners of all the cases, we present numbers of iteration 
for all the different four cases in Table~\ref{table:4case-result}.  As expected, the results of Case 2 are slightly better than the ones of Case 3, in particular for $N=32$. Although there are no remarkable differences in the presented results, in the full numerical results which are not included here, the results for Case 3 shows that this method sometimes need about 
$30 -45 \%$ more iterations than those of Case 2. In Case 4 we use the MINI element instead of the Taylor-Hood element. Although the numbers of iteration are larger than those for the Taylor--Hood element, the results are still quite robust with 
respect to  changes of parameters. As the final experiment, a model problem with nonconstant $\kappa$ is considered. We assume that $\kappa$ is small on 
\algns{
\Omega_1 = \{(x,y)\,:\, 0 \leq x \leq 1, 1/4 \leq y \leq 3/4 \} \subset \Omega, }
and $\kappa=1$ on $\Omega \setminus \Omega_1$. The numerical results in Table~\ref{table:varying-kap} are fairly robust for mesh refinements and changes of parameters, including high contrasts of $\kappa$.

\section{Conclusion}
We have studied parameter-robust discretizations and construction of preconditioners for Biot's consolidation model. To apply the framework of \cite{MR2769031} we have proposed  a new three-field formulation of the Biot system. We have showed that  preconditioners based on mapping properties and parameter-dependent norms are  robust with respect to variations of the model parameters, choice of finite element spaces satisfying the proper stability condition, and the discretization parameters. In particular, the variations of parameters in our consideration cover large shear and bulk elastic moduli, small hydraulic conductivity, small time-step, including the ranges of interest in geophysics and computational biomechanics applications. 
Furthermore, our theoretical results are confirmed  by a number of  numerical experiments.

%\dfn{Rv2-C19, need to discuss}

%\bibliography{FEM}
%\bibliographystyle{plain}

\bibliographystyle{simunsrtcompress}
\vspace{.125in}

\end{document}